\documentclass{article}

\usepackage[centertags]{amsmath}
\usepackage{amssymb}
\usepackage{mathptmx}
\usepackage{mathtools}
\usepackage[colorlinks=true,linkcolor=blue,citecolor=blue]{hyperref}
\usepackage{subfig}
\usepackage{amsthm}
\usepackage{adjustbox}
\usepackage{enumitem}

\usepackage{tikz}
\usetikzlibrary{decorations.pathmorphing}
\usepackage{graphicx, xcolor, soul}

\usepackage[colorinlistoftodos,prependcaption,textsize=tiny]{todonotes}

\DeclareMathAlphabet{\mathcal}{OMS}{cmsy}{m}{n}

\usepackage{url}
\usepackage[T1]{fontenc}

\usepackage{tikz}
\usetikzlibrary{calc}

\title{Circumcentering Reflection Methods for Nonconvex Feasibility Problems}

\author{Neil Dizon\\CARMA\\University of Newcastle \and Jeffrey Hogan\\CARMA\\University of Newcastle \and Scott B. Lindstrom\\Hong Kong Polytechnic University}

\date{\today}

\def\R{\hbox{$\mathbb R$}}
\def\N{\hbox{$\mathbb N$}}
\def\CRM{\hbox{${\rm CRM}$}}

\def\bepsilon{\hbox{$\mathbb \epsilon$}}

\def\Id{\hbox{\rm Id}}

\newcommand{\HH}{\ensuremath{\mathcal H}}

\newcommand{\Fix}{\ensuremath{\operatorname{Fix}}}

\newcommand{\calB}{\ensuremath{\mathcal B}}
\renewcommand{\vec}[1]{\mathbf{#1}}

\newtheorem{theorem}{Theorem}[section]

\newtheorem{lemma}[theorem]{Lemma}
\newtheorem{proposition}[theorem]{Proposition}
\theoremstyle{definition}
\newtheorem{definition}[theorem]{Definition}
\newtheorem{example}{\it Example}
\newtheorem{remark}[theorem]{Remark}

\begin{document}

\date{\today}

\maketitle

\begin{abstract}
	Recently, circumcentering reflection method (CRM) has been introduced for solving the feasibility problem of finding a point in the intersection of closed constraint sets. It is closely related with Douglas--Rachford method (DR). We prove local convergence of CRM in the same prototypical settings of most theoretical analysis of regular nonconvex DR, whose consideration is made natural by the geometry of the phase retrieval problem. For the purpose, we show that CRM is related to the method of subgradient projections. For many cases when DR is known to converge to a feasible point, we establish that CRM locally provides a better convergence rate. As a root finder, we show that CRM has local convergence whenever Newton--Raphson method does, has quadratic rate whenever Newton--Raphson method does, and exhibits superlinear convergence in many cases when Newton--Raphson method fails to converge at all. We also obtain explicit regions of convergence. As an interesting aside, we demonstrate local convergence of CRM to feasible points in cases when DR converges to fixed points that are not feasible. We demonstrate an extension in higher dimensions, and use it to obtain convergence rate guarantees for sphere and subspace feasibility problems. Armed with these guarantees, we experimentally discover that CRM is highly sensitive to compounding numerical error that may cause it to achieve worse rates than those guaranteed by theory. We then introduce a numerical modification that enables CRM to achieve the theoretically guaranteed rates. Any future works that study CRM for product space formulations of feasibility problems should take note of this sensitivity and account for it in numerical implementations.
\end{abstract}

\paragraph{Mathematics Subject Classification (MSC 2020):} 90C26 \and 65K10 \and 47H10 \and 49M30


\section{Introduction}\label{s:intro}

The Douglas--Rachford method (\emph{DR}) is frequently used to solve \emph{feasibility problems} of the form 
\begin{equation}\label{feasibility_problem}
	\text{find }\ \vec{x} \in A \cap B,
\end{equation}
where, here and throughout, $A$ and $B$ are closed subsets of a finite dimensional Hilbert space $\HH$ and $A \cap B \neq \emptyset$. For such problems the method consists of iterating the DR operator, which is an averaged composition of two over-relaxed projection operators defined as follows:
\begin{equation}\label{def:DRoperator}
	T_{A,B}:= \frac{1}{2} R_BR_A+\frac{1}{2}\Id \quad \text{with}\quad R_C:=2P_C-\Id,
\end{equation}
where, here and throughout, $\Id$ is the identity map and the projection map $P_S$ is as defined below in \eqref{def:P}. The DR operator owes its colloquial name to its indirect introduction in the context of nonlinear heat flow problems \cite{DR}, though it was independently discovered by Fienup in the nonconvex setting of \emph{phase retrieval} \cite{fienup1982phase}, and so it has been known under various other names \cite{LSsurvey}.

Its broader versatility in the nonconvex context was highlighted in by Elser and Gravel \cite{GE}, who applied it to solve various nonconvex combinatorial problems modeled as feasibility problems with stochastic constraints. The method has since been applied to a host of other discrete feasibility problems, including Sudoku puzzles \cite{ABT1,ABT1b}, matrix completion \cite{ABT2,BT}, graph coloring \cite{AC,artacho2018enhanced}, and bit retrieval \cite{elser2018complexity}, among others. For a more comprehensive overview of its history, including the broader context of DR as a splitting method in solving optimization problems, see for example, \cite{LSsurvey}. For more on the use of DR for solving both nonconvex and convex feasibility problems, refer to \cite{artacho2019douglas}.

The aforementioned seminal work of Elser and Gravel \cite{GE} piqued the interests of Borwein and Sims, who in 2011 made the first rigorous attempt at analysing the behaviour of DR in the nonconvex setting of hypersurfaces \cite{BS}. The spiraling convergence pattern they observed characterizes performance when DR is applied to many other nonconvex hypersurface feasibility problems, which we recall in Section~\ref{s:hypersurfaces}. 

In the convex setting, the idea of \emph{circumcentering} with the reflections has been recently introduced \cite{bauschke2018circumcentermappings,bauschke2018circumcenters,behling2018linear,circumcentering}. Other methods have also been designed, based on using the past performance to predict future iterates \cite{lindstrom2020computable,poon2019trajectory}. What motivates the \emph{present} work is the tendency of splitting methods to solve \emph{nonconvex} problems \cite{AB,BLSSS,BS,Franklin,LLS,LSS}, and the appetizing prospect of accelerating convergence in the nonconvex setting. 

For the more general monotone inclusion problem, Douglas--Rachford method is dual to the \emph{Alternating Direction Method of Multipliers} (\emph{ADMM}) \cite{Gabay,LSsurvey}. This motivated the recent introduction of a class of novel methods \cite{lindstrom2020computable} that, for some feasibility problems, includes CRM. The author introduced this class with the motivation of building algorithms that are primal/dual implementable. The simplest of them, $L_T$, may also be used for a feasibility problem. We will include it for comparison in our computed examples; for some problems, it exhibits apparently quadratic convergence, which is certainly interesting. However, as $L_T$ was derived through a somewhat different framework to what we develop here, we will not make a theoretical comparison to it.

\subsubsection*{Goal}

The goal of the present work is to furnish local convergence analysis for CRM in related nonconvex settings to those considered for DR, namely, the case when one set is a hyperplane and the other is a hypersurface that can be represented locally by the graph of a function \cite{AB,Benoist,BLSSS,BS,DT,LSS}. The purpose and value of this investigation (and those others like it) is, ultimately, not to develop a superior root finder on $\R$, but rather to understand the performance of CRM more generally. Since the construction of a single step of CRM is always computed in a 2-dimensional subspace, the study of 2-dimensional problems offers us such insights, as 2-dimensional problems are often prototypical of the 2-dimensional slices of problems in higher dimensions. This is why they have been used so often in the literature. As a natural example, we furnish Examples~\ref{ex:10spherehyperplane} and \ref{ex:10sphereline}, which show that our results on plane curves admit quadratic convergence rate guarantees for spheres and subspaces. When our numerical experiment in Example~2 at first achieves only a linear rate, we then \textit{know}, because a quadratic rate is \textit{guaranteed} by theory, that the observed linear rate is a consequence of very small numerical errors. Armed with this knowledge, we locate the source of the error, and compensate for it to recover the superior convergence rate guaranteed by the theory. The discovery of this sensitivity to small numerical error will be invaluable to further efforts to study CRM with the usual product space formulation of the feasibility problem, for reasons we explain.

\subsection{Outline}

The remainder of this paper is outlined as follows. In Section~\ref{s:preliminaries}, we provide the preliminaries on Douglas--Rachford and CRM. In Section~\ref{s:CT}, we introduce a modified version of CRM that is generically proper, and show that it has a reasonable fixed point property (Proposition~\ref{prop:fixed}). In Section~\ref{s:hypersurfaces}, we show that CRM, in the case when the sets are a hyperplane and the graph of a function on $\R^{\eta}$, is related to subgradient projection on a lower dimensional space (Theorem~\ref{thm:subgradientdescent}). In $\R^2$, the subgradient projection method is just Newton--Raphson method. In Section~\ref{s:planecurves}, we establish local convergence of CRM to a feasible point in the hypersurface settings considered previously for DR, along with convergence rate guarantees that are quadratic in many cases (Theorem~\ref{thm:main}). Our analysis, in cases with quadratic convergence, exploits the connection with Newton--Raphson method. More interestingly, in many cases when Newton--Raphson method fails to converge at all, we use the generalized angle bisector theorem to show local superlinear convergence of CRM (Lemma~\ref{lem:concave_infinity}) in $\mathbb{R}^2$. In Section~\ref{s:Rn}, we provide some examples to show one way in which our rate guarantees from Section~\ref{s:planecurves} may extend to guarantees for certain problems in $\R^\eta$. Specifically, we show quadratic convergence rates for spheres and affine subspaces, the problem that has gained particular interest for being prototypical of phase retrieval. We also provide numerical evidence that the convergence rate for CRM when one set is a subspace can be particularly sensitive to compounding numerical error, and we explain how we overcame this sensitivity to achieve the theoretical rate. We conclude in Section~\ref{s:conclusion}.

\section{Preliminaries}\label{s:preliminaries}

Splitting methods such as the Douglas--Rachford method (DR) are frequently used to solve feasibility problems as in \eqref{feasibility_problem}. When the sets of interests are closed, we often employ iterative algorithms that make use of the projector operator
$$\mathbb{P}_S\vec{x} := \left \{ \vec{z} \in S : \|\vec{x} - \vec{z}\| = \inf_{\vec{v} \in S}\|\vec{x} - \vec{v}\|\right \}.$$
In the nonconvex setting, $\mathbb{P}_S$ is a set-valued map where image values may contain more than one point. For the cases we will consider in this paper, $\mathbb{P}_S$ is always nonempty, and we simplify the exposition by working with a selector 
\begin{equation}\label{def:P}
	P_S:\HH \rightarrow S: \vec{x} \mapsto P_S\vec{x} \in \mathbb{P}_S\vec{x}.
\end{equation}

The classical result for the feasibility problem using DR, when $A$ and $B$ are closed and convex, is a consequence of a more general result of Lions \& Mercier \cite{LM}. We provide the feasibility-specific version.

\begin{theorem}[Lions \& Mercier \cite{LM}]\label{thm:LionsandMercier} Let
	$A,B \subset \HH$ be closed, convex, and nonempty, with $A \cap B \neq \emptyset$ and the sum of their normal cone operators $N_A + N_B$ also maximal monotone. For any $\vec{x}_0\in \HH$, the sequence given by $\vec{x}_{n+1}=T_{A,B}\vec{x}_n$ converges to some $\vec{v}\in \HH$ as $n\rightarrow \infty$ such that $P_A \vec{v} \in A \cap B$.
\end{theorem}
Note that the condition that $N_A+N_B$ be monotone may be relaxed for the convex feasibility problem \cite[Fact~5.9]{BCL}. The operator described in \eqref{def:DRoperator} is called the Douglas--Rachford operator.

Generically, the fixed points of this operator may not themselves be feasible; see \cite[Figure 5b]{LSsurvey} for a pictorial example. This example also illustrates the useful fact that, even in the nonconvex context, fixed points satisfy (see~\cite[Proposition 3.1]{LSsurvey})
\begin{equation}\label{eqn:fixedpointsDR}
	(\vec{x} \in \Fix T_{A,B}) \implies  \mathbb{P}_A\vec{x} \cap A \cap B \neq \emptyset.
\end{equation}

Various extensions of DR to accommodate $N$-set (where $N>2$) feasibility problems have also been considered. The most utilized is the product space reformulation \cite{GE,Pierra}. With such a reformulation, convergence is guaranteed for $N$ closed and convex sets (with nonempty intersection) by the two set result in Theorem~\ref{thm:LionsandMercier}. An augmented discussion of the above details are given in the recent survey article \cite{LSsurvey}.

\subsection{Circumcentering}\label{s:circumcentering}

Given three points $\vec{u},\vec{v},\vec{w} \in \HH$, we denote the circumcenter by $C(\vec{u},\vec{v},\vec{w})$ and define it to be the point equidistant to $\vec{u},\vec{v},\vec{w}$, and lying on the affine subspace they generate.

It may be readily verified that when $\vec{u},\vec{v}$ and $\vec{w}$ are not colinear, $C(\vec{u},\vec{v},\vec{w})$ is the intersection of the perpendicular bisectors of the sides of the triangle formed by $\vec{u},\vec{v}$ and $\vec{w}$. Figure~\ref{fig:circumcenter} illustrates that $C(\vec{u},\vec{v},\vec{w})$ is not necessarily contained within the convex hull of the triangle formed by $\vec{u},\vec{v},$ and $\vec{w}$.

\begin{figure}
	\begin{center}
		\begin{tikzpicture}[scale=2.0]
			\draw [fill,black] (-1,1) circle [radius=0.015];
			\node [above left] at (-1,1) {$\vec{u}$};
			
			\draw [fill,black] (0,0) circle [radius=0.015];
			\node [below left] at (0,0) {$\vec{v}$};
			
			\draw [fill,black] (1,0) circle [radius=0.015];
			\node [below right] at (1,0) {$\vec{w}$};		
			
			\draw [gray] (-1,1) -- (0,0) -- (1,0) -- (-1,1);
			
			\draw [blue] (-1,0) -- (1,2);
			
			\draw [blue] (0.5,-0.25) -- (0.5,2);
			
			\draw [blue] (-0.25,0) -- (.75,2);
			
			\draw [fill,purple] (0.5,1.5) circle [radius=0.030];
			\node [above left,purple] at (0.5,1.5) {$C(\vec{u},\vec{v},\vec{w})$};
		\end{tikzpicture}
	\end{center}
	\caption{The circumcenter of a triangle.}\label{fig:circumcenter}
\end{figure}

When $\{\vec{u},\vec{v},\vec{w}\}$ has cardinality 1, the definition clearly implies that $C(\vec{u},\vec{v},\vec{w})=\vec{u}=\vec{v}=\vec{w}$. When $\{\vec{u},\vec{v},\vec{w}\}$ has cardinality 2, $C(\vec{u},\vec{v},\vec{w})$ is the average of the two distinct points.

When ``circumcentering'' the reflections, we compute a new iterate by taking the circumcenter of $\vec{x},R_A\vec{x}$, and $R_BR_A\vec{x}$. Following \cite{behling2019convex}, we denote this operation by
\begin{equation}
	\CRM(\vec{x}) := C(\vec{x},R_A\vec{x},R_BR_A\vec{x}).
\end{equation}
The case where $\vec{x},R_A\vec{x},$ and $R_BR_A\vec{x}$ are distinct and colinear does not occur when the sets in question are intersecting affine subspaces as in \cite{behling2018linear,circumcentering}. It may also be avoided when the product space method is used for convex feasibility problems as in \cite{behling2019convex}. In such a case, the operator $\CRM$ is said to be \emph{proper}. Sufficient conditions for the operator $\CRM$ to be proper are given in \cite{bauschke2018circumcentermappings,bauschke2018circumcenters}. For a convex example where the operator $\CRM$ is not proper, let $A$ be the unit ball in $\R^2$, $B=\{(\lambda,3/4)\;|\; \lambda \in \R\}$, and $\vec{x} = (0,2)$. Then $R_A \vec{x} = (0,0)$ and $R_BR_A\vec{x} = (0,3/2)$, and so $\vec{x}, R_A\vec{x}, R_BR_A\vec{x}$ are distinct and colinear.

For explicit formulas for computing the circumcenter, and for a generalization of the circumcenter of a triangle to the circumcenter of sets containing finitely many points, see \cite{bauschke2018circumcenters}. For extensions of linear convergence results to infinite dimensional spaces, refer to \cite{bauschke2019linear}. A primal/dual centering approach that does not make use of reflection substeps was introduced in \cite{lindstrom2020computable}.

\section{A generically implementable nonconvex adaptation}\label{s:CT}

For nonconvex feasibility problems, $\CRM$ generically fails to be proper, and so we must choose a reasonable definition for the mapping in the nonconvex setting. This is important to ensure numerical stability in the computation, because applying the circumcentering operator to colinear (or nearly colinear) substeps may result in errors that disrupt computation. For a nonconvex problem, we may not know when this colinearity or near-colinearity might occur. The definition, therefore, should rely on conditions that are easy to check from a computational standpoint. 

Fortunately, a clear choice presents itself. When $\vec{x},R_A\vec{x}$, and $R_BR_A\vec{x}$ are colinear, the possibilities are as follow.
\begin{enumerate}[label=(\roman*)]
	\item\label{colinear1} $R_B\vec{x} = R_A\vec{x} = \vec{x}$, in which case $\vec{x} \in A \cap B \cap \Fix T_{A,B}$ and $\CRM(\vec{x}) = T_{A,B}\vec{x}$.
	\item\label{colinear2} $R_BR_A\vec{x} \neq R_A\vec{x}=\vec{x}$ or $R_BR_A\vec{x}=R_A\vec{x} \neq \vec{x}$, in which case the average of the two distinct points is just $\frac12 \vec{x} + \frac12 R_BR_A\vec{x} = T_{A,B}\vec{x}$, and so again $\CRM(\vec{x}) = T_{A,B}\vec{x}$.
	\item\label{colinear3} $R_BR_A\vec{x} = \vec{x} \neq R_A\vec{x}$, in which case $\vec{x} \in \Fix T_{A,B}$, and so $P_A\vec{x} \in A \cap B$. In this case, $\CRM(\vec{x}) = \frac{1}{2}R_A\vec{x} + \frac{1}{2}\vec{x} = P_A\vec{x} \in A \cap B$.
	\item\label{colinear4} $R_BR_A\vec{x},\; R_A\vec{x},$ and $\vec{x}$ are distinct, in which case $\CRM(\vec{x}) = \emptyset$ while $T_{A,B}\vec{x} \neq \emptyset$.
\end{enumerate}
Altogether, in cases \ref{colinear1} and \ref{colinear2} $\CRM(\vec{x})$ and $T_{A,B}\vec{x}$ coincide, and in case \ref{colinear3} it does not matter whether one updates with $\CRM(\vec{x})$ or $T_{A,B}\vec{x}$, since $P_A\vec{x}$ solves the feasibility problem. We choose to update with $T_{A,B}\vec{x}$ in case \ref{colinear4}, which is consistent but is also a reasonable choice, given what is known about the ``searching'' behaviour of DR for many nonconvex problems (specifically, it sometimes walks in a straight line before finding the local basin; see \cite{BLSSS}). Finally, for the sake of simplicity, we will also choose to ``update'' with $T_{A,B}\vec{x}$ in case \ref{colinear3}; in this way, our definition differs in the convex setting from that in \cite{circumcentering}, but it does not differ in a consequential way, since case \ref{colinear3} only occurs when we have already solved the problem. The complete definition of our generically proper circumcentering reflection operator is
\begin{equation}\label{def:circumcenteredDR}
	C_T:\HH \rightarrow \HH: \quad \vec{x}\mapsto \begin{cases}
		T_{A,B}\vec{x} & \text{if}\; \vec{x},\;R_A\vec{x},\;\text{and}\;R_BR_A\vec{x}\;\;\text{are colinear;}\\
		\CRM(\vec{x}) & \text{otherwise.}
	\end{cases}
\end{equation}
From a computational standpoint, this definition is easy to employ. One need only specify some small numerical tolerance $\epsilon >0$ and verify noncolinearity by checking that
\begin{equation*}
	\frac{\left | \langle \vec{x}-R_BR_A\vec{x},R_A\vec{x}-R_BR_A\vec{x} \rangle \right |}{\|\vec{x}-R_BR_A\vec{x}\| \|R_A\vec{x}-R_BR_A\vec{x}\|} < 1-\epsilon,
\end{equation*}
before using the closed form for the circumcenter from \cite[Theorems~8.4,8.5]{bauschke2018circumcenters}.

Since $C_T$ specifies to $\CRM$ when $\CRM(\vec{x})$ is proper---except in the uninteresting case of \ref{colinear3} when the feasibility problem is essentially solved---we immediately have the following fixed point result.

\begin{proposition}[Fixed points of $C_T$]\label{prop:fixed}
	If $\vec{x} \in \Fix C_T$ then $\mathbb{P}_A\vec{x} \cap A \cap B \neq \emptyset$.
\end{proposition}
\begin{proof}
	Let $\vec{x} \in \Fix C_T$. Then we have from \eqref{def:circumcenteredDR} that either $\vec{x}= T_{A,B}\vec{x}$ or $\vec{x} = \CRM(\vec{x})$. If $\vec{x}\in T_{A,B}\vec{x}$, then the result follows from \eqref{eqn:fixedpointsDR}. If $\vec{x}\in \CRM(\vec{x})$, then $\vec{x}$ is equidistant from $\vec{x}, R_A\vec{x}$, and $R_BR_A\vec{x}$, and so $\vec{x}=R_A\vec{x} = R_B R_A \vec{x}$. Thus $\vec{x} \in A \cap B$, and so $P_A\vec{x}=\vec{x} \in A \cap B$.
\end{proof}

\begin{definition}[Generically proper CRM]
	Let $\vec{x}_0 \in \HH$. Define $(\vec{x}_n)_{n \in \N}$ by
	\begin{equation}
		\vec{x}_{n+1}:= C_T \vec{x}_n,
	\end{equation}
	where $C_T$ is as in \eqref{def:circumcenteredDR}.
\end{definition}

For nonconvex problems, DR often has colinear substeps at the beginning of the search; see, for example, the example of Section~\ref{sec:planecurveexamples} item \ref{ex:colinear} or the example of the ellipse and line from \cite{BLSSS}. However, for many problems, this co-linear case never occurs (e.g. Examples~\ref{ex:10spherehyperplane}, \ref{ex:10sphereline}, and \ref{ex:ODE} from this paper), or at least never occurs when the algorithm starts sufficiently near to a solution (e.g. the ellipse and line example in \cite{BLSSS}). For this reason, the local analysis of convergence for $C_T$ often reduces to the analysis of CRM. For our part, we do not view ``Generically proper CRM'' to be a distinct algorithm from CRM, casewise definition notwithstanding. We simply consider it to be the natural way to adapt CRM for a nonconvex feasibility problem. Unless one knows a local theoretical guarantee that the co-linear case will not occur (as we use in our convergence analysis in Section~5), one simply \textit{must}, for computational safety, include a co-linearity check and have a plan in place for handling the co-linear case.

\section{Hypersurface feasibility problems}\label{s:hypersurfaces}

For DR, Borwein and Sims considered in detail the case of a unit sphere $A$ in $\R^{\eta}$ and a line $B$ \cite{BS}. Based on experimentation with the dynamical geometry software \emph{Cinderella} \cite{cinderella2}, they hypothesized global convergence of the sequence for starting points not on the \emph{singular set} -- the line perpendicular to $B$ and passing through the center of $A$. Arag\'{o}n Artacho and Borwein later provided a conditional proof based on the piecewise study of regions \cite{AB}, and Benoist showed convergence definitively by constructing a Lyapunov function \cite{Benoist}. 

With the proof of Benoist \cite{Benoist}, the case of a 2-sphere and a line was mostly resolved, though Borwein and Sims' conjecture of chaos on the singular manifold \cite{BS} was later disproven in the seemingly different context of \cite{BDL18}, where Bauschke, Dao, and Lindstrom proved it to be aperiodic but fully describable in terms of generalized Beatty sequences. Two generalizations of the 2-sphere were considered in \cite{BLSSS}. In this setting, the singular set has nonzero measure, and the dynamical system is characterized by basins of varying periodicities. 

\begin{figure}[ht]
	\begin{center}
		\begin{tikzpicture}[scale=1]
			\node[anchor=south west,inner sep=0] (image) at (0,0) {\includegraphics[angle=0,height=.5\textwidth]{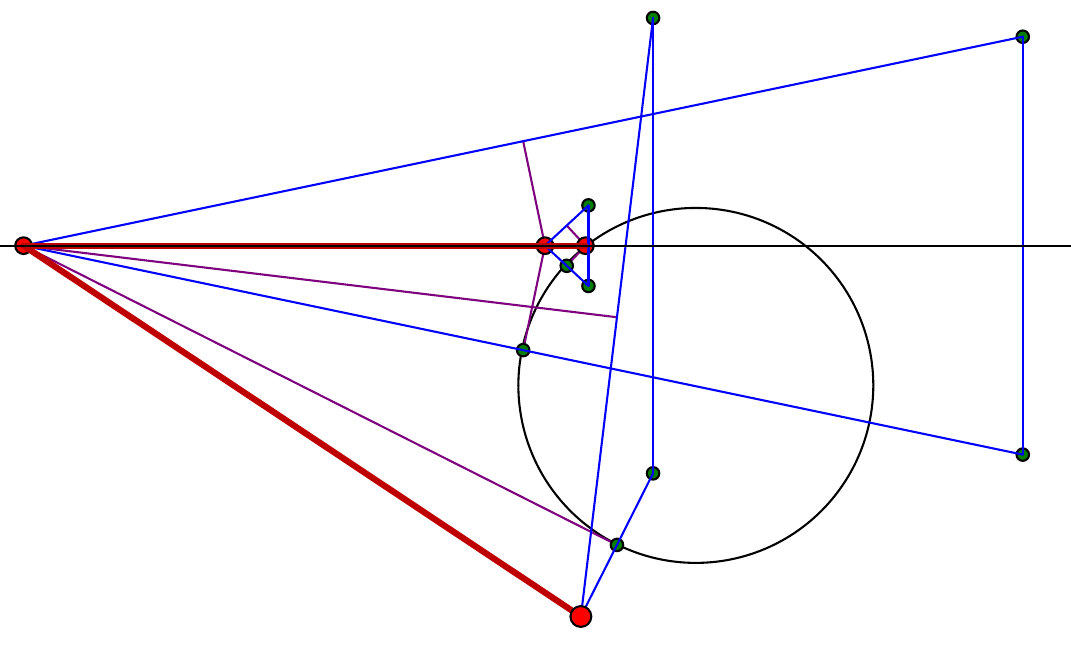}};
			\begin{scope}[x={(image.south east)},y={(image.north west)}]
				\node[black] at (0.575,0.95) {$\vec{x}_0$};
				\node[black] at (0.03,1-0.33) {$\vec{x}_1$};
				\node[black] at (0.48,1-0.33) {$\vec{x}_2$};
				\node[black] at (0.575,1-0.32) {$\vec{x}_3$};
			\end{scope}
		\end{tikzpicture}
		\caption{Several iterates of CRM for a sphere and line.}\label{fig:circle_line_circumcentered}
	\end{center}
\end{figure}

In contradistinction, as long as $\vec{x}_0$ is not in the singular set of measure zero, \emph{CRM} exhibits global convergence for any configuration of an ellipse and line, and spiralling is entirely absent. This is actually implicitly proven in \cite[see Remark~1]{behling2019convex}, because the CRM sequence, after the first update, coincides with the CRM sequence for the convex problem of the ball and a hyperplane. Figure~\ref{fig:circle_line_circumcentered} is a representative example of what the behaviour looks like, and it suggests the following result.

\begin{lemma}\label{lem:alwaysindiagonal}
	Let $\vec{x} \in \HH$. The following hold. 
	\begin{enumerate}[label=(\roman*)]
		\item\label{p41a2} If $A$ is a hyperplane and $\vec{x},R_A\vec{x},R_BR_A\vec{x}$ are not colinear, then $C_{T}\vec{x} \in A$;
		\item\label{p41a} If $B$ is a hyperplane and $\vec{x},R_A\vec{x},R_BR_A\vec{x}$ are not colinear, then $C_{T}\vec{x} \in B$;
		\item\label{p41b} If $B$ is a hyperplane and $\vec{x}=R_A\vec{x}$, then $C_{T}\vec{x} \in B$;
		\item\label{p41d} If $B$ is an affine subspace and $\vec{x} \in B$ and $\vec{x},R_A\vec{x},R_BR_A\vec{x}$ are not colinear, then $C_T \vec{x} \in B$.
	\end{enumerate}
\end{lemma}
\begin{proof}
	\ref{p41a2}: Since $\vec{x},R_A\vec{x}$, and $R_BR_A\vec{x}$ are not colinear, $C_T\vec{x} = \CRM(\vec{x})$, and so $C_T\vec{x}$ is equidistant from $R_A\vec{x}$ and $R_BR_A\vec{x}$. Since $R_A\vec{x}$ and $R_BR_A\vec{x}$ are distinct and $A$ is a hyperplane, we first claim that $A$ is the set of points equidistant from $R_A\vec{x}$ and $R_BR_A\vec{x}$. To see why, we assume without loss of generality that $\vec{0} \in A$, to simplify our notation. We can decompose $\vec{x}$ into components that are in $A$ and its orthogonal complement $A^\perp$, that is, $\vec{x}=\vec{z}_A+\vec{z}_{A^\perp}$. This implies that $P_A \vec{x}- \vec{z}_A$, and so $R_A\vec{x}= \vec{z}_A-\vec{z}_{A^\perp}$. For any point $\vec{y} \in A$, the Pythagorean theorem yields
	$$
	\|\vec{y}-\vec{x}\|^2 = \|\vec{y}- \vec{z}_A\|^2 + \|\vec{z}_{A^\perp}\|^2 = \|y - R_A\vec{x}\|^2.
	$$
	Since $\vec{y}$ is arbitrary, then any point $\vec{y}$ is equidistant from $\vec{x}$ and $R_A\vec{x}$. Suppose now that $\vec{y} \in \HH$ is equidistant from $\vec{x}$ and $R_A\vec{x}$. We may similarly write $\vec{y}=\vec{y}_A+\vec{y}_{A^\perp}$ and obtain that
	\begin{align}
		\|\vec{y}_A-\vec{z}_A\|^2 + \|\vec{y}_{A^\perp}-\vec{z}_{A^\perp}\|^2 &= \|\vec{y}_A+\vec{y}_{A^\perp}-\vec{z}_A-\vec{z}_{A^\perp}\|^2 = \|\vec{y}-\vec{x}\|^2 \stackrel{(\star)}{=} \|\vec{y}-R_A\vec{x}\|^2\nonumber \\ 
		&= \|\vec{y}_A+\vec{y}_{A^\perp}-\vec{z}_A+\vec{z}_{A^\perp}\|^2 = \|\vec{y}_A - \vec{z}_A\|^2 + \|\vec{y}_{A^{\perp}}+\vec{z}_{A^\perp}\|^2, \nonumber\\
		\text{and\;so} \quad \|\vec{y}_{A^\perp}-\vec{z}_{A^\perp}\|^2 &= \|\vec{y}_{A^\perp}+\vec{z}_{A^\perp}\|^2.\label{e:perpA}
	\end{align}
	Here $(\star)$ uses the equidistance assumption. Notice that \eqref{e:perpA} forces $\langle \vec{z}_{A^\perp},\vec{y}_{A^\perp}\rangle =0$, and since $A^\perp$ is one-dimensional, then $\vec{z}_{A^\perp}=\vec{0}$ or $\vec{y}_{A^\perp}=\vec{0}$. If $\vec{z}_{A^\perp}=\vec{0}$, then $\vec{x} = R_A\vec{x}$, which is a contradiction. Therefore $\vec{y}_{A^\perp}=\vec{0}$, and so $\vec{y} \in A$. Altogether, the set of equidistant points is exactly $A$, and so $C_T\vec{x} \in A$.
	
	\ref{p41a}: Since $\vec{x},R_A\vec{x}$, and $R_BR_A\vec{x}$ are not colinear, $C_T\vec{x} = \CRM(\vec{x})$, and so $C_T\vec{x}$ is equidistant from $R_A\vec{x}$ and $R_BR_A\vec{x}$. Since $R_A\vec{x}$ and $R_BR_A\vec{x}$ are distinct and $B$ is a hyperplane, we first claim that $B$ is the set of points equidistant from $R_A\vec{x}$ and $R_BR_A\vec{x}$. To see why, we assume without loss of generality that $\vec{0} \in B$, to simplify our notation. We can decompose $R_A\vec{x}$ into components that are in $B$ and in its orthogonal complement $B^{\perp}$, that is, $R_A\vec{x} = \vec{z}_B + \vec{z}_{B^\perp}$. This implies that $P_BR_A\vec{x}=\vec{z}_B$, and so $R_BR_A\vec{x} = \vec{z}_B - \vec{z}_{B^\perp}$ and $P_BR_BR_A\vec{x} = \vec{z}_B$. For any point $\vec{y} \in B$, the Pythagorean theorem yields
	\begin{equation*}
		\|\vec{y}-R_A\vec{x}\|^2 = \|\vec{y}-\vec{z}_B\|^2 + \|\vec{z}_{B^\perp}\|^2 = \|\vec{y}-R_BR_A\vec{x}\|^2.
	\end{equation*}
	Since $\vec{y}$ is arbitrary, then any point $\vec{y} \in B$ is equidistant from $R_A\vec{x}$ and $R_BR_A\vec{x}$. Suppose now that $\vec{y} \in \HH$ is equidistant from $R_A\vec{x}$ and $R_B R_A\vec{x}$. We may similarly write $\vec{y} = \vec{y}_B + \vec{y}_{B^\perp}$ and obtain that
	\begin{align}
		\|\vec{y}_B-\vec{z}_B\|^2 + \|\vec{y}_{B^\perp}-\vec{z}_{B^\perp}\|^2 &= \|\vec{y}_B+\vec{y}_{B^\perp}-\vec{z}_B-\vec{z}_{B^\perp}\|^2 = \|\vec{y}-R_A\vec{x}\|^2 \stackrel{(\star)}{=} \|\vec{y}-R_BR_A\vec{x}\|^2\nonumber \\ 
		&= \|\vec{y}_B+\vec{y}_{B^\perp}-\vec{z}_B+\vec{z}_{B^\perp}\|^2 = \|\vec{y}_B - \vec{z}_B\|^2 + \|\vec{y}_{B^{\perp}}+\vec{z}_{B^\perp}\|^2, \nonumber\\
		\text{and\;so} \quad \|\vec{y}_{B^\perp}-\vec{z}_{B^\perp}\|^2 &= \|\vec{y}_{B^\perp}+\vec{z}_{B^\perp}\|^2.\label{e:perp}
	\end{align}
	Here $(\star)$ uses the equidistance assumption. Notice that \eqref{e:perp} forces $\langle \vec{z}_{B^\perp},\vec{y}_{B^\perp}\rangle = 0$, and since $B^\perp$ is one-dimensional then $\vec{z}_{B^\perp}=\vec{0}$ or $\vec{y}_{B^\perp}=\vec{0}$. If $\vec{z}_{B^\perp}=\vec{0}$, then $R_A\vec{x} = R_BR_A\vec{x}$, which is a contradiction. Therefore $\vec{y}_{B^\perp}=\vec{0}$, and so $\vec{y} \in B$. Altogether, the set of equidistant points is $B$, and so $C_T\vec{x} \in B$.

	\ref{p41b}: Since $\vec{x}=R_A\vec{x}$, we have $C_{T}\vec{x} = T_{A,B}\vec{x}$ which expands to
	$$
	C_{T}\vec{x}=(1/2)\vec{x}+(1/2)R_B R_A\vec{x} = (1/2)\vec{x}+(1/2)R_B \vec{x} = P_B\vec{x},
	$$
	where the second equality uses the condition that $\vec{x}=R_A\vec{x}$. Thus, $C_{T}\vec{x} \in B$.
	
	\ref{p41d}: Let $L$ be the set of points that are equidistant from $R_A\vec{x}$ and $R_BR_A\vec{x}$. To show that $C_T(\vec{x}) \in B,$ it suffices to show that 
	$$
	L \cap {\rm aff}(\vec{x},R_A\vec{x},R_BR_A\vec{x}) \subset B.
	$$
	Because $\vec{x},R_A\vec{x}, R_BR_A\vec{x}$ are not colinear, it must be true that $R_A\vec{x} \neq R_BR_A\vec{x}$, and so the inclusion
	$$
	L \cap {\rm aff}(\vec{x},R_A\vec{x},R_BR_A\vec{x}) \subsetneq {\rm aff}(\vec{x},R_A\vec{x},R_BR_A\vec{x})
	$$
	is not an equality. Combining with the fact that the affine subspace on the right side is 2-dimensional, we have
	$$
	{\rm dim}(L \cap {\rm aff}(\vec{x},R_A\vec{x},R_BR_A\vec{x}))=1.
	$$
	Notice that in our proof of \ref{p41a}, the argument that shows that $B$ is contained in the subset of equidistant points (i.e. $B \subset L$) depended only upon $B$ being an affine subspace (only the reverse inclusion $\supset$ relied upon $B$ being a hyperplane) and so $B \subset L$ holds now by the same argument. More specifically, since $\vec{x} \in B$ and $P_BR_A\vec{x} \in B$, we have that
	$$
	B^*:={\rm aff}\{\vec{x},P_BR_A\vec{x}\} \subset B \subset L.
	$$
	Of course, we also have
	$$
	B^* \subset {\rm aff}\{\vec{x},R_A\vec{x},R_BR_A\vec{x} \}.
	$$
	Moreover, since $\vec{x},R_A\vec{x}$, and $R_BR_A\vec{x}$ are not colinear $B^*$ is of dimension 1. Altogether, we have that $B^*$ is a one-dimensional affine subspace contained in $L \cap {\rm aff}(\vec{x},R_A\vec{x},R_BR_A\vec{x})$. As the two affine subspaces are of dimension 1, and one is contained in the other, they must be equal. This concludes the result.
\end{proof}

Of the above proposition, the following should be noted:
\begin{enumerate}
    \item Items \ref{p41a2} and \ref{p41a} of the above proposition are not necessarily true if the set that is assumed to be a hyperplane is, instead, only assumed to be an affine subspace of arbitrary dimension. In case \ref{p41a}, for instance, if one assumes $B$ is only an affine subspace instead of a hyperplane, then one can only obtain $\vec{y}_{B^\perp}\in\{\vec{z}_{B^\perp} \}^\perp$. As an example, the perpendicular bisector of a reflection across a line in $\R^3$ is a hyperplane containing the line; the circumcenter will be in the hyperplane, but not necessarily the line.
    \item Item \ref{p41d} is already known in the case when $A$ is convex, due to \cite[Lemma 3]{behling2019convex}. That result uses in its proof the convexity of $A$; if it did not, then we would simply cite the extension as vacuous.
\end{enumerate}

Throughout the remaining part of this section, we write $\vec{x}=(x,x') \in \R^{\eta} \times \R=\R^{\eta+1}$ where $x$ is the component of $\vec{x}$ in $\R^{\eta}$ and $x'$ is the component of $\vec{x}$ in $\R$. Furthermore, we set $A,B \subset \R^{\eta+1}$ as
$$
A = {\rm gra}f = \{(y,f(y))\;|\;y \in {\rm dom} f \subset \R^{\eta} \}\quad \text{and} \quad B = \R^{\eta} \times \{0\}.
$$
Here, ${\rm gra } f$ denotes the \emph{graph} of a function $f:\R^{\eta} \to \R$. Furthermore, we use $\partial^0 f$ to denote the symmetric subdifferential of $f$, which is given by $\partial^0f:=\partial f \cup (-\partial (-f))$, where $\partial f(x):=\{x^* \in X\;|\; (x^*,1) \in N_{{\rm epi}f}(x,f(x)) \}$ is the \textit{limiting subdifferential} of $f$ at $x$ \cite[Section~2.3]{DT}. The following lemma establishes a relationship between CRM and subgradient projections in $\R^{\eta}$.

\begin{theorem}\label{thm:subgradientdescent}
	Let $\vec{x}=(x,x') \in \R^{\eta} \times \R$, $B=\{(t,0) : t\in \R^{\eta}\}$ and $A={\rm gra }f$ where $f:\R^{\eta} \rightarrow \R$ is proper and has a closed graph. Suppose further that $f$ is Lipschitz continuous locally at $y$ where $(y,f(y))=:P_A\vec{x}$. Then the following hold.
	\begin{enumerate}[label=(\roman*)]
		\item\label{case:notcolinear} If $\vec{x},R_A\vec{x},R_BR_A\vec{x}$ are not colinear, then
		\begin{align*}
			C_T\vec{x} &= \left(y-\frac{f(y)}{\|y^*\|^2}y^*,0 \right)\\
			\text{where}\quad & y^* \in \partial^0(f(y)) \;\; \text{satisfies}\;\; x=y+(f(y)-x')y^*
		\end{align*}
		\item\label{case:colinear} Otherwise, $C_T\vec{x} = \left(y,x' - f(y) \right)$.
	\end{enumerate}
\end{theorem}

\begin{proof}
	First note that the existence of $y^*$ that satisfies 
	\begin{equation}\label{eqn:handy}
		x=y+(f(y)-x')y^*
	\end{equation}
	is assured by \cite[Lemma~3.4]{DT}.
	
	\ref{case:notcolinear}: The assumption of non-colinearity forces $C_T\vec{x}=\CRM(\vec{x})$. We must have $y^* \neq 0$, since $(y^*=0) \implies (x=y)$, which forces $\vec{x},R_A\vec{x},R_BR_A\vec{x}$ to be colinear, a contradiction. Similarly note also that $(f(y)-x'=0)\implies (x=y)$, which forces the same contradiction; therefore $f(y)-x' \neq 0$. Any point $\vec{w}=(w,w')$ in the perpendicular bisector of $(\vec{x},R_A\vec{x})$ satisfies
	\begin{align*}
		0&=\left \langle \vec{x}-P_A\vec{x},(w,w')-P_A\vec{x} \right \rangle\\
		&=\left \langle (x-y,x' -f(y)),(w-y,w'-f(y)) \right \rangle\\
		&= \left \langle \left((f(y)-x')y^*,x'-f(y) \right), (w-y,w'-f(y)) \right \rangle\quad \text{(using\;\eqref{eqn:handy})}\\
		&=(f(y)-x')\left \langle \left((y^*,-1 \right), (w-y,w'-f(y)) \right \rangle\\
		&\stackrel{(a)}{=}(f(y)-x') \left(\langle (w-y),y^*\rangle -w' + f(y)) \right)\\
		\iff (w,w') &\stackrel{(b)}{=} (y+\lambda,f(y)+\langle \lambda, y^* \rangle) \;\;\text{for\; some}\;\; \lambda \in \R^{\eta}.
	\end{align*}
	To see why (b) holds, remember that $f(y)-x' \neq 0$, and so the equality (a) is equivalent to $\left(\langle (w-y),y^*\rangle -w' + f(y) \right)=0$, which is equivalent to $w'=f(y)+\langle w-y,y^*\rangle$. Defining $\lambda:=w-y$, (b) and (a) are equivalent.
	
	Altogether, we have shown that the perpendicular bisector of the segment $(\vec{x},R_A\vec{x})$ is the tangent plane $(y+\lambda,f(y)+\langle \lambda, y^*\rangle )_{\lambda \in \R^{\eta}}=:H$. Combining with Lemma~\ref{lem:alwaysindiagonal} and the definition of CRM, we have that $\CRM(\vec{x})\in H \cap B \cap {\rm aff}(\vec{x},R_A\vec{x},R_B R_A\vec{x})$. We also have that
	\begin{equation}\label{e:affineequality}
		{\rm aff}(\vec{x},R_A\vec{x},R_B R_A\vec{x}) = \{(y+\alpha y^*,\beta)\;\;|\;\; (\alpha,\beta) \in \R^2\}.
	\end{equation}
	To see why \eqref{e:affineequality} holds, let $\vec{v} \in {\rm aff}(\vec{x},R_A\vec{x},R_B R_A\vec{x})$. Then 
	\begin{subequations}
	\begin{align}
	\vec{v}=& \alpha_{\vec{x}} \vec{x} + \alpha_{R_A\vec{x}} R_A\vec{x} + \alpha_{R_BR_A\vec{x}} R_BR_A\vec{x} \quad \text{with}\quad \alpha_{\vec{x}}+\alpha_{R_A\vec{x}}+\alpha_{R_BR_A\vec{x}}=1 \label{useconditions1}\\
	=&\alpha_{\vec{x}} (x,x') + \alpha_{R_A\vec{x}} (2y-x,2f(y)-x') + \alpha_{R_BR_A\vec{x}} (2y-x,-2f(y)+x')\label{useconditions2}\\
	=&\left(y+\alpha^*y^*,\beta^* \right)\label{useconditions3}\\
	\text{where} \quad \alpha^*:=&(2\alpha_{\vec{x}}-1)(f(y)-x')\label{useconditions4}\\
	\text{and}\quad \beta^*:=&\alpha_{\vec{x}}x'+(2f(y)-x')(\alpha_{R_A\vec{x}}-\alpha_{R_B R_A\vec{x}}).\label{useconditions5}
	\end{align}
	\end{subequations}
	Here \eqref{useconditions1} simply uses the definition of the affine hull; \eqref{useconditions2} uses the definition of $B$, whereby $R_BR_A\vec{x}$ is obtained by taking $R_A\vec{x}$ and simply reversing the sign on its $(\eta+1)$th coordinate; \eqref{useconditions3} uses the fact that $x=y+(f(y)-x')y^*$ and the equality $\alpha_{\vec{x}}+\alpha_{R_A\vec{x}}+\alpha_{R_BR_A\vec{x}}=1$. Altogether, we have shown $\vec{v}=(y+\alpha^*,\beta^*) \in \{(y+\alpha y^*,\beta)\;\;|\;\; (\alpha,\beta) \in \R^2\}$, showing the inclusion $\subset$ in \eqref{e:affineequality}. Moreover, as both affine subspaces in \eqref{e:affineequality} are of the same dimension and one is contained in the other, they must be equal, and so the equality holds in \eqref{e:affineequality}.
	
	Altogether, we have
	\begin{align*}
		{\rm aff}(\vec{x},R_A\vec{x},R_B R_A\vec{x}) \cap H \cap B &=\left\{\underbrace{(y+\lambda,f(y)+\langle \lambda,y^* \rangle)}_{\text{constraint $H$}}\;|\; \underbrace{\lambda = \alpha y^*}_{\text{constraint}\; {\rm aff(\dots)}}, \underbrace{f(y)+\langle \lambda,y^*\rangle =0}_{\text{constraint $B$}} \right\}\\
		&=\left \{\left(y-\frac{f(y)}{\|y^*\|^2}y^*,0\right)\right \}.
	\end{align*}
	This shows the desired result.
	
	\ref{case:colinear}: In the colinear case, $C_T\vec{x}=T_{A,B}\vec{x}$. Computing, one has $R_A\vec{x} = (2y-x,2f(y)-x')$, and $R_BR_A\vec{x} = (2y-x,-2f(y)+x')$, and finally $(1/2)R_B R_A\vec{x}+(1/2)\vec{x} = (y,x'-f(y))$, which shows the result.
\end{proof}

\begin{remark}[A subgradient projections characterization]\label{rem:subgradientprojections}
	Notice that  $y \mapsto y-\frac{f(y)}{\|y^*\|^2}y^*$ is a step of subgradient projections applied to the function $f$ at $y$. For a differentiable function  $f:\R \to \R$, it reduces to a step of Newton--Raphson method: $y \mapsto y-f(y)/f'(y)$ where $y \in \R$. Thus, Theorem~\ref{thm:subgradientdescent} shows that, for problems when $C_T\vec{x}=\CRM(\vec{x})$ and $\vec{x} \in \R^2$, the method reduces to a step of alternating projections $P_B P_A \vec{x}$, followed by a step of subgradient descent for the function $f$. Figure~\ref{fig:sqrtX} illustrates this in the case when $A$ is the graph of the function $x \mapsto x/\sqrt{|x|} \; (x\in\R)$ and $B$ is the horizontal axis. Solving the feasibility problem amounts to finding the root of the function $x \mapsto x/\sqrt{|x|} \; (x\in\R)$, a problem that Newton's method on $\R$ fails to solve.
\end{remark}

Interestingly, \cite[Proposition~3.6]{lindstrom2020computable} has described a relationship between CRM and the method of subgradient projections applied for a different function: one that describes the dynamical system admitted by repeated application of the DR operator. This function, called a Lyapunov function, is defined on a lifted space. The author used the theory of Lyapunov functions to motivate their introduction of different algorithm, $L_T$, which we will include for comparison in our computed examples.

The use of DR to find roots of functions on $\R$, as well as for $x/\sqrt{|x|} \; (x\in\R)$ in particular, was first considered in \cite{LSS}. Dao and Tam adapted Benoist's Lyapunov function approach in order to demonstrate local convergence for non-tangentially intersecting cases thereof \cite{DT}.

\begin{figure}
	\begin{center}
		\begin{tikzpicture}[scale=1]
			\node[anchor=south west,inner sep=0] (image) at (0,0) {\includegraphics[width=.75\textwidth]{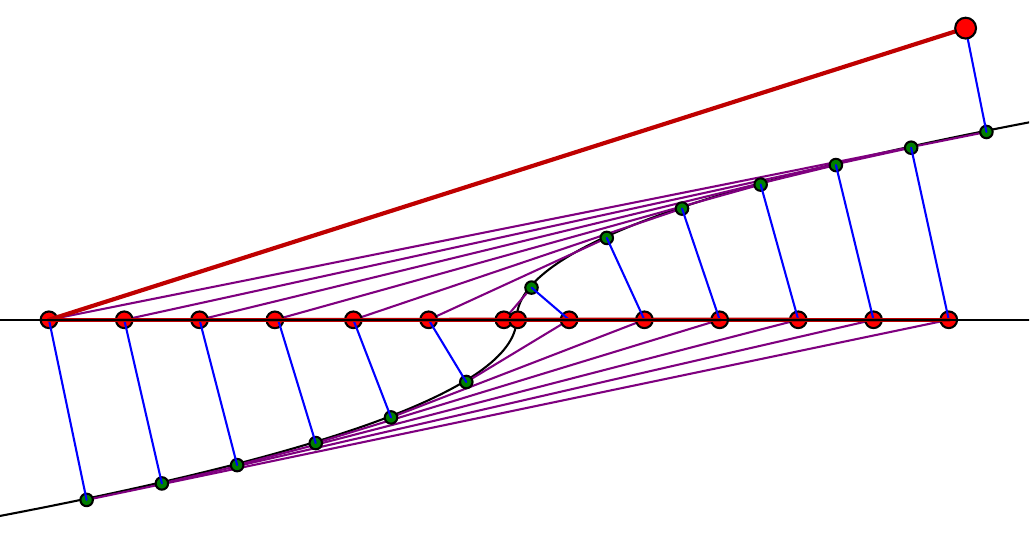}};
			\begin{scope}[x={(image.south east)},y={(image.north west)}]
				
				\node[black] at (0.95,1) {$\vec{x}_0$};
				\node[black] at (0.04,0.35) {$\vec{x}_1$};
				\node[black] at (0.11,0.35) {$\vec{x}_3$};
				\node[black] at (0.18,0.35) {$\vec{x}_5$};
				\node[black] at (0.25,0.35) {$\vec{x}_7$};
				\node[black] at (0.32,0.35) {$\vec{x}_9$};
				\node[black] at (0.40,0.35) {$\vec{x}_{11}$};
				\node[black] at (0.57,0.46) {$\vec{x}_{12}$};
				\node[black] at (0.65,0.46) {$\vec{x}_{10}$};
				\node[black] at (0.73,0.46) {$\vec{x}_{8}$};
				\node[black] at (0.81,0.46) {$\vec{x}_{6}$};
				\node[black] at (0.89,0.46) {$\vec{x}_{4}$};
				\node[black] at (0.96,0.46) {$\vec{x}_{2}$};
			\end{scope}
		\end{tikzpicture}
	\end{center}
	\caption{CRM applied to find a root of $x\mapsto x/\sqrt{|x|} \mbox{ where } (x\in\R)$}\label{fig:sqrtX}
\end{figure}

\section{A line and a plane curve}\label{s:planecurves}

Most local convergence results in the nonconvex setting of hypersurfaces have focused on hyperplanes and graphs of functions \cite{AB,Benoist,BLSSS,BS,DT,LSS} with $\R^2$ as a setting of particular focus. One motivation is that the \textit{phase retrieval problem} may be thought of as a feasibility problem on $n$-tuples of $\R^2$. Another reason is that DR has variously been observed to spiral in lower dimensional subspaces, a phenomenon theorized to often occur in a lower dimensional affine subspace of $\R^{\eta}$. For example, Arag\'on Artacho has created an image that shows this behaviour for a line and sphere in $\R^3$ \cite[slide 33]{Borweinlectures}. For CRM, $\R^2$ is especially a natural context for investigation because the circumcenter construction for $\vec{x}$ is performed in the affine hull of $\vec{x},R_A\vec{x},$ and $R_BR_A\vec{x}$, which in the non-colinear case is always a copy of $\R^2$. Thus, sets in $\R^2$ generically represent slices of sets in the affine hull of reflection substeps for higher dimensional problems. Similar planar results motivated, in part, the introduction of the first primal/dual centering method in \cite{lindstrom2020computable}.

Under mild assumptions, we will show local convergence in $\R^2$ for the algorithm generated by iteratively applying the map $C_T$ with a curve $A$ and line $B$. Our approach differs from those in previous works both in terms of the methods used and the results obtained. In particular, the use of the generalized angle bisector theorem suggests a path forward for more complicated problems.

\subsubsection*{Comparing approaches}

The first theoretical result on the convergence of DR in the nonconvex setting is that of Borwein and Sims \cite{BS} who used Perron theorem \cite[Theorem~6.1]{BS} \cite[Corollary~4.7.2]{Lak&DT} on the stability of almost linear difference equations to show local convergence in the setting where $A$ is a unit sphere in $\R^{\eta}$ and $B$ is a line. This approach has since been adapted \cite{BLSSS} to show local convergence for plane curves more generally. The strategy relies upon the fact that $T_{A,B}\vec{x}$ may be described as a continuous function of the 3-tuple $(\vec{x},R_A\vec{x},R_B R_A \vec{x})$. This continuity does not extend to the case of the operator $C_{T}$, and so the same approach does not immediately extend to this new context.

The approach of Benoist, Dao, and Tam relies on a Lyapunov function \cite{Benoist,DT,lindstrom2020computable}. It is less clear how to adapt such an approach when the spiral itself is obviated by circumcentering the method.

Our approach is to use Theorem~\ref{thm:subgradientdescent} and employ trigonometry to show results about the intersections of the tangents taken for the curve $A$ with the line $B$. Without loss of generality, we let $L$ be the horizontal axis. We then use the observation in Remark~\ref{rem:subgradientprojections} that, under mild assumptions, $C_{T}$ locally behaves like a step of alternating projections for $A$ and $B$ followed by a step of Newton--Raphson method employed to find a root of the function whose graph is the curve $A$. Figure~\ref{fig:convergence} is helpful in understanding this observation. Here, when $B$ is the horizontal axis and $A$ is the graph of a function $f:\R \to \R$, we may associate the points $\vec{x}_n,\vec{x}_{n+1},$ and $(y_{n},f(y_n))$ in $\R^2 = \{\vec{t}=(t,t')\;|\; t,t' \in \R \}$ with their horizontal components $x_n,x_{n+1},$ and $y_{n}$. We then have the relationship
\begin{equation}\label{d:NewtonRaphson}
	x_{n+1} = g_f(y_n) := y_n-\frac{f(y_n)}{f'(y_n)}, \quad \text{where}\;\; (y_n,0)=P_B P_A(\vec{x}_n).
\end{equation}
Here, $g_f$ is the Newton--Raphson operator that is commonly used to search for a root of the function $f$, and for which convergence results are well known. For the sake of cleanliness in Figure~\ref{fig:convergence}, we abuse notation slightly by assigning the label $g_f(x_n)$ to the point that is actually $(g_f(x_n),0)$. Note that when $f$ is strictly differentiable at $y \in \R$, $\partial^0 f(y) = f'(y)$ \cite[Corollary~1.82]{mordukhovich2006variational}, an identity we will implicitly use in this section whenever appealing to Theorem~\ref{thm:subgradientdescent}.

\subsection{Auxiliary results on Newton--Raphson method}

Because our results will exploit the connection with Newton--Raphson method, we will need two preliminary lemmas that tailor the classical theory to our purpose.

\begin{lemma}[When Newton--Raphson rate is quadratic]\label{lem:NR}
	Let $\theta:\R \rightarrow \R$ and $g(t) =t-\theta(t)/\theta'(t)$ for $t \in \left[-\bepsilon,\bepsilon \right]$ and $\theta(0) = 0$. Suppose there exist $h_{\rm LEFT},h_{\rm RIGHT} \in C^2[-\bepsilon,\bepsilon]$
	such that
	\begin{align*}
		(\forall t \in \left[0,\bepsilon\right])\quad h_{\rm RIGHT}(t) = \theta(t)\quad \text{and}\quad (\forall t \in \left[-\bepsilon,0\right])\quad h_{\rm LEFT}(t) = \theta(t),
	\end{align*}
	with $h'_{\rm LEFT}(0) , h'_{\rm RIGHT}(0) \in \R \setminus \{0\}$. Then there exists $\delta>0$ such that for $p_0 \in \left[-\delta,\delta \right]$, the sequence defined by $p_n = g(p_{n-1})$, when $n\geq 1$, converges at least quadratically to $0$. Moreover, for $t$ sufficiently near $0$,
	$$
	|g(t)| < \frac{M}{2}|t|^2,
	$$
	for some $M \in \R$.
\end{lemma}
\begin{proof}
	By the classical result on Newton--Raphson (see \cite[Theorem~2.9]{BF16})
	, there exist $\delta_{\rm LEFT}$, $\delta_{\rm RIGHT}$, $M_{\rm LEFT}$, $M_{\rm RIGHT}$ such that
	\begin{align*}
		t \in \left[-\delta_{\rm LEFT},\delta_{\rm LEFT} \right] \quad &\implies \quad |g_{\rm LEFT}(t)| < \frac{M_{\rm LEFT}}{2}|t|^2\\
		\text{and}\;\;t \in \left[-\delta_{\rm RIGHT},\delta_{\rm RIGHT} \right] \quad &\implies \quad |g_{\rm RIGHT}(t)| < \frac{M_{\rm RIGHT}}{2}|t|^2\\
		\text{where}\quad g_{\rm RIGHT}:t &\rightarrow t-h_{\rm RIGHT}(t)/h'_{\rm RIGHT}(t)\\ \text{and}\quad  g_{\rm LEFT}:t &\rightarrow t-h_{\rm LEFT}(t)/h'_{\rm LEFT}(t).
	\end{align*}
	Letting $M_{\rm MAX}:=\max\{M_{\rm LEFT},M_{\rm RIGHT} \}$ and $\delta_{\rm MIN}:=\min\{\delta_{\rm LEFT},\delta_{\rm RIGHT} \}$ and using the fact that
	\begin{equation*}
		g(t) = \begin{cases}
			g_{\rm LEFT}(t) & \text{if}\;\; t\leq 0\\
			g_{\rm RIGHT}(t) & \text{if}\;\; t \geq 0,
		\end{cases}
	\end{equation*}
	we have that 
	\begin{equation}\label{e:deltamin}
		t \in \left[-\delta_{\rm MIN},\delta_{\rm MIN} \right] \quad \implies \quad |g(t)| < \frac{M_{\rm MAX}}{2}|t|^2.
	\end{equation}
\end{proof}

\begin{lemma}[When Newton--Raphson rate is linear]\label{lem:NRlinear}
	Let $h \in C^m\left[-\bepsilon,\bepsilon \right]$ and $p=0$ be a root of $h$ of multiplicity $m$. Let $g:t \mapsto t-h(t)/h'(t)$. Then there exists $k<1$ and $\delta_{\rm LIN}>0$ such that 
	$$
	t \in \left[-\delta_{\rm LIN},\delta_{\rm LIN} \right]\quad \implies \quad |g(t)-p| < k|t-p|.
	$$
\end{lemma}
\begin{proof}
	Using the fact that
	$$
	g'(t) = \frac{h(t)h''(t)}{h'(t)^2},
	$$
	we have that
	\begin{align*}
		g'(p) &= \underset{t \rightarrow 0}{\lim} \frac{h(p+t)h''(p+t)}{h'(p+t)^2}\\
		&= \underset{t \rightarrow 0}{\lim} \frac{\left(\sum_{j=0}^m \frac{h^j(p)}{j!}t^j +O(t^{m+1})\right)\left(\sum_{j=2}^m \frac{h^j(p)}{(j-2)!}t^{j-2} +O(t^{m-1}) \right)}{\left(\sum_{j=1}^m \frac{h^j(p)}{(j-1)!}t^{j-1} +O(t^m)\right)^2}\\
		&\stackrel{(\star \star)}{=}\underset{t \rightarrow 0}{\lim} \frac{\frac{h^m(p)^2}{m!(m-2)!}t^{2m-2}+O(t^{2m-1})}{\frac{h^m(p)^2}{(m-1)!^2}t^{2m-2}+O(t^{2m-1})}\\
		&=\frac{(m-1)!^2}{m!(m-2)!} = \frac{m-1}{m}<1.
	\end{align*}
	Notice that $(\star \star)$ uses the fact that $h^j(p)=0$ for $j<m$ (since $p$ is a root of $h$ of multiplicity $m$). Letting $k = \frac{1}{2}(\frac{m-1}{m}+1)$, we have that, for $t$ sufficiently near to $p$, $g'(t)< k <1$. In particular, we may choose $\delta_{\rm LIN}$ so that 
	$$
	t \in \left[-\delta_{\rm LIN},\delta_{\rm LIN} \right] \quad \implies \quad g'(t)<k.
	$$
	Applying \cite[Theorem~2.8]{BF16}
	, we have that
	$$
	t \in \left[-\delta_{\rm LIN},\delta_{\rm LIN} \right] \quad \implies |g(t)-p| < k|t-p|,
	$$
	which shows the result.
\end{proof}

\subsection{Convergence of CRM: basic conditions}\label{s:convergence}

Throughout the rest of this section, $A$ is the graph of a function $f:\R\to\R$, $B$ is the horizontal axis $\R \times \{0\} \subset \R^2$,  $\vec{x}=(x,x') \in \R^2$, $(y,f(y)) := P_A\vec{x}$,  and $\calB_{0}(r):=\left [-r,r \right ]$. As the projection onto $B$ is computable, one can always start with the first iterate $\vec{x}_0 \in B$. In particular, we will show that for $\vec{x}_0 \in B$ started sufficiently near to the solution, the iterated scheme $\vec{x}_n:=(C_T)^n\vec{x}_0$ reduces to $\vec{x}_n:=(\CRM)^n\vec{x}_0$ and converges to the solution $\vec{0}$. In particular, the conditions we will impose on $f$ will locally prevent colinearity of $\vec{x},R_A\vec{x}$ and $R_BR_A \vec{x}$ (as will be shown in Lemma~\ref{lem:Newton--Raphson}) so that we always have Theorem~\ref{thm:subgradientdescent}\ref{case:notcolinear}, which simplifies our analysis greatly. We will consider the case when the following hold.
\begin{enumerate}[label=(\Roman*)]
	\item\label{bc1} the curve $A$ is the graph of a continuous function $f$ and where $\vec{0} \in A \cap B$ is an isolated feasible point;
	\item\label{bc2} there exists $\epsilon_1 >0$ such that $0$ is the only root of $f$ on $\calB_{0}(\epsilon_1)$;
	\item\label{bc3} $f$ is continuous on $\calB_{0}(\epsilon_1)$ and differentiable on $\calB_{0}(\epsilon_1)\setminus \{0 \}$; 
	\item\label{bc4} there exists $\epsilon_2>0$ such that ${\rm zer}f' \cap \calB_{0}(\epsilon_2) \subset \{0\}$, that is, $f'$ has no other roots in $\calB_0(\epsilon_2)$ except possibly for $0$.
\end{enumerate}
The isolated root condition we have imposed on $f$ excludes such pathological cases as $f: t \mapsto \sin(1/t)$ \big(where $t\in\R\backslash \{0\}$\big), while the analogous condition we have imposed on $f'$ further excludes such pathological cases as $f: t \mapsto t \sin (1/t)$ \big(where $t\in\R\backslash \{0\}$\big). Finally, we will assume a similar condition about $f''$ namely
\begin{enumerate}[label=(\Roman*)]\setcounter{enumi}{4}
	\item\label{bc5} there exists $\epsilon_3>0$ such that $f'$ is continuous and differentiable on $\calB_0(\epsilon_3)\setminus \{0\}$ with ${\rm zer} f'' \cap \calB_0(\epsilon_3) \subset \{0\}$.
\end{enumerate}
In other words, $f''$ has no other roots in $\calB_0(\epsilon_3)$ except, possibly, for $0$. Since the sign of $f$ does not change on $\left ]0,{\bepsilon}_1 \right]$, we may by symmetry work in the first quadrant by assuming (without loss of generality) further that
\begin{enumerate}[label=(\Roman*)]\setcounter{enumi}{5}
	\item\label{bc6} $f$ is positive on $\left ]0,{\bepsilon}_1 \right]$.
\end{enumerate}

When $f$ satisfies the conditions \ref{bc1}--\ref{bc6}, we write $\epsilon_f := \min\{\epsilon_1,\epsilon_2/2,\epsilon_3 \}$. The reason for the choice of $\epsilon_2/2$ will be apparent in Lemma~\ref{lem:Newton--Raphson}.  Since $f$ is positive on $\left ]0,{\bepsilon}_f \right]$ and $f'$ does not have a root on $]0,\epsilon_f]$, then we further have that $f$ is increasing on $\left ]0,{\bepsilon}_f \right]$, or equivalently, $f'$ is positive on $\left ]0,{\bepsilon}_f \right]$. 

\begin{remark}[When $A$ is a line segment locally]
	Our conditions have excluded the possibility that there exists $\delta >0$ such that $f''$ is zero on  $\left[0,\delta \right]$. Of course, in this case, the graph of $A$ is a line segment locally, and so we have convergence in a single step, by Theorem~\ref{thm:subgradientdescent}. Our use of the graph of a single-valued function $f$ to represent $A$ also precludes the possibility that $A$ is locally a vertical line segment, another case when local convergence is immediate. Thus, we do not lose any (new) insights by disregarding such cases.
\end{remark}

\subsection{Convergence of CRM}

In what follows, $\vec{x}_+:=C_T\vec{x}$. For the supposed conditions in the following Lemma~\ref{lem:Newton--Raphson}, if one remembers that we are only writing $x\geq 0$ because we may do so by symmetry and without loss of generality, then the conditions we have assumed, in essence, are implied by $\vec{x}$ lying in $B$ and within a certain local ball about the solution, a standard assumption used for studying local convergence.

\begin{lemma}[Newton--Raphson as a convergence rate bound]\label{lem:Newton--Raphson}
	Suppose $f$ satisfies the basic conditions \ref{bc1}--\ref{bc6} and that $\vec{x}=(x,0)$ with $0<x \leq \epsilon_f$. Then, whenever $\vec{0} \neq (y,f(y))=P_A\vec{x}$, the following hold.
	\begin{enumerate}[label=(\roman*)]
		\item\label{lem:newtonraphson1} $y \in \left]0,x\right[$
		\item\label{lem:newtonraphson2} $C_T \vec{x} = (g_f(y),0)$
		\item\label{lem:newtonraphson3} If $f''(t) > 0$ for $t \in \left]0,\epsilon_3\right]$, then $0 \leq g_f(y) < g_f(x)$;
		\item\label{lem:newtonraphson4} If $f''(t) < 0$ for $t \in \left]0,\epsilon_3\right]$, then $g_f(x) < g_f(y) \leq 0$;
		\item\label{eqn:NR} $\frac{\|\vec{x}_+\|}{\|\vec{x}\|} = \frac{|g_f(y)|}{|x|}\leq \frac{|g_f(x)|}{|x|}$.
	\end{enumerate}
\end{lemma}
\begin{proof}
	\ref{lem:newtonraphson1} Since $x \leq \epsilon_2/2$ and $(0,0) \in A$, one has $\|\vec{x}-P_A(\vec{x})\| \leq \|(x,0)-(0,0)\|= x \leq \epsilon_2/2$. Since $(x,f(x)) \in A$, one also has $\|\vec{x}-P_A(\vec{x})\|\leq \|(x,0)-(x,f(x))\| = f(x)$. It then follows that
	\begin{equation}\label{lem:newtonraphson1contra}
		\|\vec{x}-P_A(\vec{x})\|\leq \min \{x,f(x)\} \leq \epsilon_2/2.
	\end{equation}
	Suppose, for a contradiction, that $P_A(\vec{x}) = (t,f(t))$ for $t \notin \left]0,x \right]$. If $t>\epsilon_2$, then $\|\vec{x}-P_A(\vec{x})\| =\|(x-t,-f(t))\| \geq \|t-x\| > \epsilon_f$ which contradicts \eqref{lem:newtonraphson1contra}. Thus $t \leq \epsilon_2$. Since $f$ is monotone increasing on $\left[0,\epsilon_2\right]$, any point $(t,f(t)) \in A$ that satisfies $\epsilon_2>t>x$ must also satisfy $f(t)> f(x)$, which forces $\|\vec{x}-P_A(\vec{x}) \| = \|(x,0)-(t,f(t))\| > f(t) > f(x)$ contradicting \eqref{lem:newtonraphson1contra}. Lastly, if $t<0$, one has $\|\vec{x}-P_A(\vec{x}) \| > x$, which is again a contradiction of \eqref{lem:newtonraphson1contra}. Thus, we have established that $y \in \left]0,x \right]$. Consequently, we also have that $f$ is Lipschitz continuous on an open ball about $y$. Since $(y,f(y)) = P_A((x,0))$, we may apply \cite[Lemma~3.4]{DT} to obtain \eqref{eqn:handy}, which simplifies to
	$$
	x = y+f(y)f'(y).
	$$
	If $f(y)=0$, then $y=0$, which is a case we may ignore. Otherwise, we have $f(y)f'(y) >0$, and so $x > y$. Thus $y \in \left]0,x\right[$.
	
	\ref{lem:newtonraphson2}: Having shown \ref{lem:newtonraphson1}, we have that 
	$$
	R_A(\vec{x}) = (2y-x,2f(y)),\;\; R_BR_A(\vec{x}) = (2y-x,-2f(y)), \;\;\text{and}\;\; \vec{x} = (x,0)
	$$
	are not colinear. Thereupon applying Theorem~\ref{thm:subgradientdescent}, we obtain \ref{lem:newtonraphson2}.
	
	\ref{lem:newtonraphson3}: By the mean value theorem, there exists $c \in \left]y,x\right[$ with
	\begin{equation}\label{lem:newtonraphson3a}
		(x-y)f'(c)=f(x)-f(y).
	\end{equation}
	Combining with the fact that $f''(t)> 0$ for $t \in \left[y,x\right]$, we have that $f'$ is monotone increasing on $\left[y,x\right]$, and so $f'(y) < f'(c) < f'(x)$. Combining with \eqref{lem:newtonraphson3a},
	\begin{align}
		(x-y)f'(x) &> f(x)-f(y),\nonumber \\
		\text{and\;so}\quad x-y -\frac{f(x)}{f'(x)} &> -\frac{f(y)}{f'(x)} > -\frac{f(y)}{f'(y)},\nonumber \\
		\text{which\;shows}\quad  x-\frac{f(x)}{f'(x)} &> y -\frac{f(y)}{f'(y)}.\label{lem:newtonraphson3b}
	\end{align}
	Applying the mean value theorem again, there exists $c' \in \left]0,y\right[$ such that
	\begin{equation}\label{lem:newtonraphson3c}
		(y-0)f'(c')=f(y)-f(0)
	\end{equation}
	Since $f'$ is monotone increasing on $\left[0,y\right]$, we have $f'(c') < f'(y)$. Combining this fact with \eqref{lem:newtonraphson3c}, we have $y f'(y) \geq f(y)$, and so 
	\begin{equation}\label{lem:newtonraphson3d}
		0 \leq y-f(y)/f'(y).
	\end{equation}
	Together, \eqref{lem:newtonraphson3b} and \eqref{lem:newtonraphson3d} show \ref{lem:newtonraphson3}. Combining with 	\ref{lem:newtonraphson2}, we have
	\begin{alignat*}{2}
		\vec{x}_+=(x_+,0), \quad \text{and}\quad 0 \leq x_+ = g_f(y) \leq g_f(x),
	\end{alignat*}
	which together yield \ref{eqn:NR}.
	
	\ref{lem:newtonraphson4}: The proof is similar to \ref{lem:newtonraphson3}, with the only change being that $f'$ is monotone \textit{decreasing} instead of \textit{increasing}, which reverses the directions of the analogous inequalities. The proof of \ref{eqn:NR} then follows in the same way.
\end{proof}

As we have already noted, because $f''$ has no roots on $\left]0,\epsilon_f \right]$, $f'$ is monotone on $\left]0,\epsilon_f \right]$. Consequently, if $f'$ is bounded on $\left]0,\epsilon_f \right]$, then we have $\lim_{t \downarrow 0} f'(t) \in \left [0, \infty \right[$ exists by the monotone convergence theorem, and if $f'$ is unbounded, then $\lim_{t \downarrow 0} f'(t) = \infty$. Taken together with Lemma~\ref{lem:Newton--Raphson}, the conditions we have assumed about $f''$ leave us with the following four natural cases to consider -- those illustrated in Figure~\ref{fig:convergence}.
\begin{enumerate}[label=$(\roman*)^\dagger$]
	\item\label{case:concave} $f''(t) < 0$ for $t \in \left[0,\epsilon_f \right]$, and $f$ is right differentiable at $0$ with
	$$
	\lim_{t \downarrow 0}f'(t) \in \left]0,\infty \right[.
	$$
	\item\label{case:concave_infinity} $f''(t) < 0$ for $t \in \left[0,\epsilon_f \right]$, and
	$$
	\lim_{t \downarrow 0}f'(t) = \infty.
	$$
	\item\label{case:convex}$f''(t) > 0$ for $t \in \left[0,\epsilon_f \right]$, and $f'(0) \neq 0$.
	\item\label{case:convex_zero}$f''(t) > 0$ for $t \in \left[0,\epsilon_f \right]$, and $f'(0) = 0$ and there exists $h \in C^m\left[-\epsilon_f,\epsilon_f \right]$ such that $h(t)=f(t)$ for $t>0$ and $0$ is a root of $h$ of multiplicity $m$ (for example, any function analytic at $0$ whose Taylor series has $0$ as a root of multiplicity greater than 1 will meet this criterion). 
\end{enumerate}

Lemma~\ref{lem:Newton--Raphson} essentially shows that the convergence rate of Newton--Raphson serves as an upper bound on the convergence rate of CRM in many cases. Now we use the generalized angle bisector theorem to obtain superlinear results in cases when Newton--Raphson cycles. Here, $\angle (\vec{a},\vec{b},\vec{c})$ denotes, as usual, the angle at which the segments $\vec{a}\vec{b}$ and $\vec{b}\vec{c}$ meet.

\begin{figure}[ht]
	\begin{center}
		\subfloat[Case~\ref{case:concave}]{\begin{tikzpicture}[scale=3]
				\draw[scale=1.0,domain=0:1,smooth,variable=\x,blue] plot ({0.5*(\x+0.5)*(\x+0.5)-0.125},{\x});
				
				
				\draw [black] (-0.5,0) -- (1,0);
				\draw [gray,dashed] (0,0) -- (0,1);
				
				\draw [fill,black] (0.5,0) circle [radius=0.015];
				\node [below] at (0.5,0) {$(y,0)$};
				
				\draw [fill,black] (0.5,.625) circle [radius=0.015];
				\node [right] at (0.5,.625) {$P_A\vec{x}$};
				
				\draw[black,dotted] (0.5,.575) -- (0.5,0.05);
				
				\draw [fill,black] (-0.191,0) circle [radius=0.015];
				\node [below] at (-0.191,0) {$\vec{x}^+$};
				\draw [black,->] (0.45,.575) -- (-0.141,0.05);
				
				\draw [fill,black] (1,0) circle [radius=0.015];
				\node [below] at (1,0) {$\vec{x}$};
				\draw [black,->] (0.95,0.05) -- (0.55,.575);

				\draw [fill,red] (.1825,0) circle [radius=0.015];
				\node [below,red] at (.1825,0) {$\vec{\tau}$};
				\draw [red,dashed] (.1825,0) -- (0.5,.625);
				\node [above right,red] at (.2,0) {$\varphi$};
				
				\draw [fill,purple] (1,1) circle [radius=0.015];
				
				\draw [fill,purple] (-0.5,0) circle [radius=0.015];
				\node[below,purple] at (-0.5,0) {$g_f(x)$};
				
				\draw[dashed, purple,->] (.95,0.965) -- (-0.45,0.045);
				\draw[dashed, purple,<-] (1,0.95) -- (1,0);

			\end{tikzpicture}\label{fig:concave}}\quad
		\subfloat[Case~\ref{case:concave_infinity}]{\begin{tikzpicture}[scale=3]
				\draw[scale=1.0,domain=0:1,smooth,variable=\x,blue] plot ({\x*\x},{\x});
				
				
				\draw [black] (-0.5,0) -- (1,0);
				\draw [gray,dashed] (0,0) -- (0,1);
				
				\draw [fill,black] (1,0) circle [radius=0.015];
				\node[below] at (1,0) {$\vec{x}$};

				\draw [fill,red] (0.45,0) circle [radius=0.015];
				\node[above right,red] at (0.45,0) {$\varphi$};
				\node[below,red] at (0.45,0) {$\vec{\tau}=(y,0)$};
				
				\draw [fill,black] (0.45,.67) circle [radius=0.015];
				\node[above] at (0.45,.67) {$P_A\vec{x}$};
				
				\draw[black,->] (0.95,0.04) -- (0.5,.63);
				
				\draw [red,dashed] (0.45,.62) -- (0.45,0.05);
				
				\draw [fill,black] (-0.45,0) circle [radius=0.015];
				\node[below] at (-0.45,0) {$\vec{x}^+$};
				
				\draw[->,black] (0.4,.63) -- (-0.4,0.05);

			\end{tikzpicture}\label{fig:concave_infinity}}
		
		\subfloat[Case~\ref{case:convex}]{\begin{tikzpicture}[scale=4.5]
				\draw[scale=1.0,domain=0:1,smooth,variable=\x,blue] plot ({\x},{0.5*(\x+0.5)*(\x+0.5)-0.125});
				
				
				\draw [black] (0,0) -- (1,0);
				\draw [gray,dashed] (0,0) -- (0,1);
				
				\draw [fill,black] (1,0) circle [radius=0.015];
				\node [below] at (1,0) {$\vec{x}$};
				
				\draw [fill,black] (0.55,.42625) circle [radius=0.015];
				\node [above left] at (0.55,.42625) {$P_A\vec{x}$};
				\draw [black,->] (.95,.05) -- (0.6,.375);
				
				\draw [fill,black] (0.144,0) circle [radius=0.015];
				\node [below left] at (0.144,0) {$\vec{x}^+$};
				
				\draw [black,->] (0.5,.375) -- (0.194,0.05);
				
				\draw [fill,purple] (1,1) circle [radius=0.015];
				\draw [fill,purple] (0.333,0) circle [radius=0.015];
				\node [below,purple] at (0.333,0) {$g_f(x)$};
				
				\draw [purple,->,dashed] (0.95,0.925) -- (0.383,0.05);
				\draw [purple,->,dashed] (1,0.05) -- (1,0.95);
				
				\draw [fill,black] (.55,0) circle [radius=0.015];
				\node [below right] at (.55,0) {$(y,0)$};
				\draw [dotted] (.55,0.05) -- (.55,.375);

			\end{tikzpicture}\label{fig:convex}}\quad
		\subfloat[Case~\ref{case:convex_zero}]{\begin{tikzpicture}[scale=4.5]
				\draw[scale=1,domain=0:1,smooth,variable=\x,blue] plot ({\x},{\x*\x});
				
				\draw [black] (0,0) -- (1,0);
				\draw [gray,dashed] (0,0) -- (0,1);
				
				\draw [fill,black] (1,0) circle [radius=0.015];	
				\node [below] at (1,0) {$\vec{x}$};

				\draw [fill,black] (0.6,0.36) circle [radius=0.015];	
				\node [above left] at (0.6,0.36) {$P_A\vec{x}$};
				
				\draw [black,->] (0.95,0.05) -- (0.625,0.335);

				\draw [fill,black] (0.3,0) circle [radius=0.015];
				\node [below left]  at (0.3,0) {$\vec{x}^+$};
				
				\draw [black,->] (0.56,0.31) -- (0.35,0.05);
				
				\draw [fill,black] (0.6,0) circle [radius=0.015];
				\node [above right] at (0.6,0) {$(y,0)$};
				
				\draw [dotted] (0.6,0.05) -- (0.6,0.335);
				
				\draw [fill,purple] (1,1) circle [radius=0.015];
				
				\draw [fill,purple] (0.5,0) circle [radius=0.015];
				\node [purple,below] at (0.5,0)  {$g_f(x)$};
				
				\draw [purple,->,dashed] (0.975,0.95) -- (0.525,0.05);
				
				\draw [purple,->,dashed] (1,0.05) -- (1,0.95);

				
			\end{tikzpicture}\label{fig:convex_horizontal}}
	\end{center}
	\caption{Cases for proof of local convergence.}\label{fig:convergence}
\end{figure}

\begin{lemma}[Cases~\ref{case:concave} and \ref{case:concave_infinity} superlinear]\label{lem:concave_infinity}
	Let $f$ satisfy the basic conditions \ref{bc1}--\ref{bc6}, $f''(t) < 0$ for $t \in \left]0,\bepsilon_3 \right]$, and $\vec{x}=(x,0)$. Then for any $K>0$ there exists $\bepsilon_{K_f}>0$ such that
	\begin{equation}
		\left(x \in \left]0,\bepsilon_{K_f} \right]\right) \implies \frac{\|\vec{x}_+\|}{\|\vec{x}\|} \leq \frac{1}{K}.
	\end{equation}
\end{lemma}
\begin{proof}
	We illustrate our construction in Figure~\ref{fig:concave} and \ref{fig:concave_infinity}. We can and do assume that $0<x < \epsilon_f$ so that the characterizations in Lemma~\ref{lem:Newton--Raphson} always hold.  For the sake of simplicity, when $\lim_{t\downarrow 0} f'(t)<\infty$, we write $f'(0):=\lim_{t\downarrow 0} f'(t)$. We also denote 
	$$
	(\varphi,\vec{\tau}):=\begin{cases}
		\left(\arctan (f'(0)),\left(y-f(y)/f'(0),0\right)\right) & \text{if}\;\; \lim_{t\downarrow 0} f'(t)<0\\
		\left(\pi/2, (y,0) \right) & \text{otherwise,}
	\end{cases}
	$$
	where $\varphi$ is an angle measure and $\tau$ is a point. By considering these two cases, we will first show that 
	\begin{equation}\label{concaveX2}
		\exists \delta_f>0 \quad \text{such\;that}\quad (x \in \left]0,\delta_f \right]) \implies \cot \left( \angle (\vec{x}, \vec{x}_+, P_A\vec{x})\right) \leq \cot \left((1/2)\varphi \right).
	\end{equation}
	
	\textbf{Case:} $\lim_{t \downarrow 0} f'(t)= \infty.$
	
	Combining this assumption with the continuity of $f'$, and the fact that $f''$ is negative on $\left[0,\bepsilon \right]$, 
	\begin{equation}\label{delta_f}
		(\exists \delta_f \in \left[0,\bepsilon \right])\quad \text{such that} \quad (t \leq \delta_f) \implies f'(t) \geq 1.
	\end{equation}
	Letting $x \in \left]0,\delta_f \right]$, we have from Lemma~\ref{lem:Newton--Raphson} that $y \in \left[0,\delta_f \right]$, and so $f'(y) \geq 1$. Consequently, we have that
	\begin{alignat*}{2}
		&& \arctan(f'(y)) = \angle (\vec{x}, \vec{x}_+, P_A\vec{x}) &\geq \pi/4 \geq (1/2)(\pi/2)=(1/2)\varphi. \\
		\text{and so}\quad && \cot \left( \angle (\vec{x}, \vec{x}_+, P_A\vec{x})\right) &\leq \cot \left((1/2)\varphi \right),
	\end{alignat*}
	which shows \eqref{concaveX2}.
	
	\textbf{Case: } $\lim_{t \downarrow 0} f'(t)< \infty.$
	
	Since $f'(0)<\infty$ and $\varphi:=\arctan(f'(0))$, there exists $\delta_f \in \left[0,\bepsilon \right]$ so that	
	$$(x \leq \delta_f) \implies \arctan (f'(t)) \geq (1/2) \arctan (f'(0)) = (1/2) \varphi,$$
	whereupon $\arctan(f'(y)) \geq (1/2)\varphi$. Consequently,  
	\begin{alignat*}{2}
		&&\arctan(f'(y)) &= \angle (\vec{x},\vec{x}_+, P_A\vec{x}) \geq (1/2)\varphi.\nonumber \\
		\text{and so}\quad && \cot \left( \angle (\vec{x}, \vec{x}_+, P_A\vec{x})\right) &\leq \cot \left((1/2)\varphi \right),
	\end{alignat*}
	which again shows \eqref{concaveX2}. This concludes our need for separate cases. 
	
	Next, notice that
	\begin{subequations}\label{sube}
	\begin{align}
	    \|\vec{x}_+\| &\leq \|\vec{x}_+-\vec{\tau}\| \label{sube1}\\
	    \text{and}\quad \|\vec{x}\| &\geq \|\vec{x}-\vec{\tau}\|.\label{sube2}
	\end{align}
	\end{subequations}
	To see why, recall that $\vec{x}_+=(g_f(x),0)$ with $g_f(x)\leq0$ (Lemma~\ref{lem:Newton--Raphson}\ref{lem:newtonraphson4}). By definition, either $\vec{\tau}=(\hat{\tau},0)$ satisfies $\hat{\tau} =y \in \left]0,x \right[$ or $\hat{\tau}=y-f(y)/f'(0)$. In the case when $\hat{\tau} = y \in \left]0,x \right[$, we have 
	\begin{align*}
	\|\vec{x}_+\| = |g_f(x)| < |y|+|g_f(x)| = |y-g_f(x)|=\|\vec{x}_+-\vec{\tau} \|,& \quad \text{which shows \eqref{sube1}}\\
	\text{and}\quad \|\vec{x}\| = |x| > |x-y| = \|\vec{x}-\vec{\tau}\|,& \quad \text{which shows \eqref{sube2}}.
	\end{align*}
	In the other case when $\hat{\tau}=y-f(y)/f'(0)$, we know that $\infty >f'(0) \geq f'(y) >0$ (this is just from the definition of $\vec{\tau}$). We have from the mean value theorem that some $c \in \left[0,y\right]$ satisfies $f(y)-f(0)=f'(c)(y-0)$, which is just $y=f(y)/f'(c)$. By monotonicity, $f'(0) \geq f'(c)$. Altogether,
	$$
	\hat{\tau}=y-\frac{f(y)}{f'(0)} \geq y-\frac{f(y)}{f'(c)} = y-y=0.
	$$
	Combined with the fact that $\vec{x}_+=(g_f(x),0)$ with $g_f(x)\leq0$, we again obtain \eqref{sube1}. Noticing further that 
	$$
	\hat{\tau} = y-\frac{f(y)}{f'(0)} \leq y < x,
	$$
	we have that $\hat{\tau} \in \left[0,x\right],$ which shows \eqref{sube2}.
	
    Now we have the following.
	\begin{subequations}\label{concaveBIG}
		\begin{align}
			\frac{\|\vec{x}_+\|}{\|\vec{x}\|} &\leq \frac{\|\vec{x}_+-\vec{\tau}\|}{\|\vec{x}-\vec{\tau}\|}\label{concaveA}\\
			&= \frac{\|\vec{x}_+-P_A\vec{x}\|}{\|\vec{x}-P_A\vec{x}\|} \cdot \frac{\sin \left( \angle (\vec{x}_+, P_A\vec{x},\vec{\tau}) \right)}{\sin \left(\angle (\vec{x},P_A\vec{x},\vec{\tau}) \right)} \label{concaveB} \\
			&= \cot \left(\angle (\vec{x}, \vec{x}_+, P_A\vec{x}) \right) \cdot \frac{\sin \left(\angle (\vec{x}_+,P_A\vec{x},\vec{\tau}) \right)}{\sin \left(\angle (\vec{x},P_A\vec{x},\vec{\tau}) \right)} \label{concaveC}\\
			&\leq \cot \left((1/2)\varphi \right) \cdot \frac{\sin \left(\angle (\vec{x}_+,P_A\vec{x},\vec{\tau}) \right)}{\sin \left(\angle (\vec{x},P_A\vec{x},\vec{\tau}) \right)}. \label{concaveD}
		\end{align}
	\end{subequations}
	Here \eqref{concaveA} follows from \eqref{sube}. We have the identity \eqref{concaveB} from the generalized angle bisector theorem, 
	and \eqref{concaveC} is mere trigonometry. Finally, \eqref{concaveD} is from \eqref{concaveX2}. Now notice that
	\begin{subequations}\label{angle0}
		\begin{align}
			\angle(\vec{x}_+,P_A\vec{x},\vec{\tau}) &= \pi - \angle (\vec{x}_+,\vec{\tau},P_A\vec{x}) - \angle (P_A\vec{x},\vec{x}_+,\vec{\tau})\label{angle0a} \\
			&=\pi - (\pi -\angle(P_A\vec{x},\vec{\tau},(y,0))) - \angle(P_A\vec{x},\vec{x}_+,\vec{\tau})\label{angle0b}\\
			&=\angle(P_A\vec{x},\vec{\tau},(y,0))-\angle(P_A\vec{x},\vec{x}_+,\vec{\tau})\nonumber \\
			&= \arctan(f'(0))-\arctan(f'(y))\label{angle0d}\\
			& \downarrow 0\;\; \text{as}\;\;x \downarrow 0. \label{angle0e}
		\end{align}
	\end{subequations}
	Here \eqref{angle0a} is simply the fact that the sum of the angles of the triangle $(\vec{x}_+,P_A\vec{x},\vec{\tau})$ is $\pi$, \eqref{angle0b} uses the fact that $\angle(P_A\vec{x},\vec{\tau},(y,0))$ is complementary to $\angle (\vec{x}_+,\vec{\tau},P_A\vec{x})$, and \eqref{angle0d} simply uses the fact that $\tan(\angle (P_A\vec{x},\vec{\tau},(y,0))) = \|P_A\vec{x}-(y,0)\|/\|(y,0)-\vec{\tau}\| =f'(0)$ while $\tan(\angle(P_A\vec{x},\vec{x}_+,\vec{\tau})) = \tan(\angle(P_A\vec{x},\vec{x}_+,(y,0))) = \|P_A\vec{x}-(y,0)\|/\|(y,0)-\vec{x}_+\| = f'(y)$. Finally, \eqref{angle0e} uses continuity and the fact that $0 \leq y < x$. Additionally, we have that
	\begin{subequations}\label{anglepi}
		\begin{align}
			\angle(\vec{x},P_A\vec{x},\vec{\tau}) &= \pi/2 - \angle(\vec{x}_+,P_A\vec{x},\vec{\tau})\label{anglepia} \\
			&\uparrow \pi/2\;\;\text{as}\;\;x \downarrow 0,\label{anglepib}
		\end{align}
	\end{subequations}
	where \eqref{anglepia} uses the fact that the sum of the two angles $\angle(\vec{x},P_A\vec{x},\vec{\tau})$ and $\angle(\vec{x}_+,P_A\vec{x},\vec{\tau})$ is $\angle(\vec{x}_+,P_A\vec{x},\vec{x}) = \pi/2$, and \eqref{anglepib} uses \eqref{angle0}. Combining \eqref{angle0} and \eqref{anglepi} and the continuity of the sine function, we have that 
	\begin{equation}\label{downtozer0}
		\frac{\sin \left(\angle (\vec{x}_+,P_A\vec{x},\vec{\tau}) \right)}{\sin \left(\angle (\vec{x},P_A\vec{x},\vec{\tau}) \right)} \downarrow 0 \;\; \text{as}\;\; x \downarrow 0.
	\end{equation}
	As a consequence of \eqref{downtozer0}, for any $K>0$, we may choose $\bepsilon_{K_f}>0$ small enough to ensure that
	\begin{equation*}
		(x \leq \bepsilon_{K_f}) \implies \frac{\sin \left(\angle (\vec{x}_+,P_A\vec{x},\vec{\tau}) \right)}{\sin \left(\angle (\vec{x},P_A\vec{x},\vec{\tau}) \right)} \leq \frac{1}{\cot \left((1/2)\varphi \right)K},
	\end{equation*}
	which combines with \eqref{concaveD} to ensure that
	\begin{equation*}
		(x \leq \bepsilon_{K_f}) \quad \implies \quad \frac{\|\vec{x}_+\|}{\|\vec{x}\|} \leq \frac{1}{K},
	\end{equation*}
	This shows the desired result.
\end{proof}
Lemma~\ref{lem:concave_infinity} essentially shows local superlinear convergence rate for a setting in which Newton--Raphson method is not guaranteed to converge at all, and for settings where it has provably failed. One such example is $t \mapsto t/\sqrt{|t|}$ (where $t\in\R$). The behaviour of CRM for this problem is shown in Figure~\ref{fig:sqrtX}. 

Note that the basic conditions \ref{bc1}--\ref{bc5} in Section~\ref{s:convergence} allowed us to prove Lemmas~\ref{lem:Newton--Raphson} and \ref{lem:concave_infinity} under the assumption of condition \ref{bc6}, requiring that the graph of $f$ sits (at least locally) in the first quadrant. Of course, there is nothing special about this quadrant. We now formalize our more general results in Theorem~\ref{thm:main}. Owing to the sheer number of cases that can be reduced to our framework, Theorem~\ref{thm:main} is wide but not all-encompassing.

\begin{theorem}\label{thm:main}
	Let $\theta:\R \rightarrow \R$ be defined such that the functions $f=|\theta|$ and $f=|\theta\circ(-\Id)|$ satisfy the basic conditions \ref{bc1}--\ref{bc5}. Then the following hold.
	\begin{enumerate}[label=(\Roman*)]
		\item\label{thm:concave} Let $\theta$ be such that the functions $f=|\theta|$ and $f=|\theta\circ(-\Id)|$ satisfy \ref{case:concave}. Let $K'>1$ and 
		\begin{equation}\label{def:epsilon}
			\bepsilon = \min \left \{\epsilon_{|\theta|},\epsilon_{|\theta\circ(-{\rm Id})|},\bepsilon_{K'_{|\theta|}},\bepsilon_{K'_{|\theta\circ(-{\rm Id})|}} \right \},
		\end{equation}
		where
		$\epsilon_{|\theta|},\epsilon_{|\theta\circ(-{\rm Id})|}$ are as defined in the basic conditions for the functions $|\theta|$ and $|\theta\circ(-{\rm Id})|$ respectively, and $\bepsilon_{K'_{|\theta|}},\bepsilon_{K'_{|\theta\circ(-{\rm Id})|}}$ are as defined in Lemma~\ref{lem:concave_infinity}. Let $\vec{x}_0 = (x_0,0)$ with $x_0 \in \left[-\bepsilon,\bepsilon\right]$. Then $\vec{x}_n \rightarrow \vec{0}$ as $n \rightarrow \infty$ with convergence rate that is never worse than $1/K'$ and ultimately either finite or quadratic.
		\item\label{thm:concave_infinity} Let $\theta$ be such that the functions $f=|\theta|$ and $f=|\theta\circ(-\Id)|$ satisfy \ref{case:concave_infinity}. Let $K'>1$ and $\epsilon$ be as in \eqref{def:epsilon}. Let $\vec{x}_0 = (x_0,0)$ with $x_0 \in \left[-\bepsilon,\bepsilon\right]$. Then $\vec{x}_n \rightarrow \vec{0}$ as $n \rightarrow \infty$ with convergence rate that is never worse than $1/K'$ and ultimately either finite or superlinear.
		
		\item\label{thm:convex} If $f=|\theta|$ satisfies \ref{case:convex}, and $\vec{x}_0 = (x_0,0)$ with $x_0 \in \left[0, \min \{\epsilon_{|\theta|},\delta_{\rm RIGHT}\}\right]$ where $\delta_{\rm RIGHT}$ is as in Lemma~\ref{lem:NR}, then $\vec{x}_n \rightarrow \vec{0}$ as $n \rightarrow \infty$ with convergence rate that is ultimately either finite or quadratic.
		
		\item\label{thm:convex_zero} Let $f=|\theta|$ satisfy \ref{case:convex_zero} and $\vec{x}_0 = (x_0,0)$ with $x_0 \in \left[0, \min \{\epsilon_{|\theta|},\delta_{\rm LIN}\}\right]$ where $\delta_{\rm LIN}$ is as defined in Lemma~\ref{lem:NRlinear} for the function $h$ specified in \ref{case:convex_zero} with $f=|\theta|$. Then $\vec{x}_n \rightarrow \vec{0}$ as $n \rightarrow \infty$ with convergence rate that is ultimately either finite or linear.
	\end{enumerate}
\end{theorem}
\begin{proof}
	\ref{thm:concave_infinity}: Suppose the convergence is not finite. Combining our assumptions with Lemma~\ref{lem:concave_infinity}, 
	\begin{equation*}
		\frac{\|\vec{x}_{n+1}\|}{\|\vec{x}_n\|} \leq \frac{1}{K'},
	\end{equation*}
	which shows that $x_n \rightarrow 0$ with convergence rate that is no worse than $1/K'$. From Lemma~\ref{lem:concave_infinity} we also have that, for any $K>0$, there exist constants $\bepsilon_{K_{|\theta|}}$ and $\bepsilon_{K_{|\theta \circ ({\rm -Id}) |}}$ such that  
	\begin{equation}\label{concA}
		\left(x_n \in \left[-\bepsilon_{K_{|\theta|}},\bepsilon_{K_{|\theta|}} \right] \cap \left[-\bepsilon_{K_{|\theta\circ(-{\rm Id})|}},\bepsilon_{K_{|\theta\circ(-{\rm Id})|}} \right]\right) \implies \frac{\|\vec{x}_{n+1}\|}{\|\vec{x}_n\|}\leq \frac{1}{K}.
	\end{equation}
	Since $x_n \rightarrow 0$, we have that there exists $N_K$ sufficiently large that 
	\begin{equation}\label{concB}
		(n \geq N_K) \implies x_n \in \left[-\bepsilon_{K_{|\theta|}},\bepsilon_{K_{|\theta|}} \right] \cap \left[-\bepsilon_{K_{|\theta\circ(-{\rm Id})|}},\bepsilon_{K_{|\theta\circ(-{\rm Id})|}} \right]
	\end{equation}
	Combining \eqref{concA} and \eqref{concB}, we have that 
	\begin{equation*}
		(n \geq N_K) \implies \frac{\|\vec{x}_{n+1}\|}{\|\vec{x}_n\|} \leq \frac{1}{K},
	\end{equation*}
	which shows the claimed superlinear convergence.
	
	\ref{thm:concave}: By the same arguments we just used for \ref{thm:concave_infinity}, we have that $x_n \rightarrow 0$. Suppose the convergence is not finite. For $n$ sufficiently large, $x_n \in \left[-\delta_{\rm MIN},\delta_{\rm MIN} \right]$ where $\delta_{\rm MIN}$ is as defined in Lemma~\ref{lem:NR}, and so we may combine with Lemma~\ref{lem:Newton--Raphson}\ref{eqn:NR} to obtain
	\begin{equation*}
		\frac{\|\vec{x}_{n+1}\|}{\|\vec{x}_n\|^2} \leq \frac{|g_\theta(x_n)|}{|x_n|^2} < \frac{M_{\rm MAX}}{2},
	\end{equation*}
	where $M_{\rm MAX}$ is as given in Lemma~\ref{lem:NR}, which shows the eventual quadratic convergence rate.
	
	\ref{thm:convex}: Suppose the convergence is not finite. Because $x_n \in \left[0, \min \{\epsilon_{|\theta|},\delta_{\rm RIGHT}\}\right]$, applying Lemma~\ref{lem:NR} together with Lemma~\ref{lem:Newton--Raphson}\ref{eqn:NR}  
	yields
	\begin{equation*}
		\frac{\|\vec{x}_{n+1}\|}{\|\vec{x}_n\|^2} = \frac{|x_{n+1}|}{|x_n|^2} \leq \frac{|g_\theta(x_n)|}{|x_n|^2} < \frac{M_{\rm RIGHT}}{2},
	\end{equation*}
	which shows the quadratic convergence.
	
	\ref{thm:convex_zero}: Suppose the convergence is not finite. Using the fact that $f=|\theta|$ satisfies \ref{case:convex_zero}, we apply Lemma~\ref{lem:NRlinear} for the corresponding function $h$ and combine Lemma~\ref{lem:Newton--Raphson}\ref{eqn:NR} 
	with the equality of $f$ and $h$ on $\left[0,\min \{\epsilon_{|\theta|},\delta_{\rm LIN}\}\right]$ to obtain
	\begin{equation*}
		\frac{\|\vec{x}_{n+1}\|}{\|\vec{x}_n\|} = \frac{|x_{n+1}|}{|x_n|} \leq \frac{|g_\theta(x_n)|}{|x_n|} =  \frac{|g_h(x_n)|}{|x_n|} < k,
	\end{equation*}
	where $k = \frac{1}{2}(\frac{m-1}{m}+1)<1$. This shows the linear convergence.
\end{proof}

\subsection{Plane curve examples}\label{sec:planecurveexamples}

Local convergence of CRM to a feasible point for a line $B$ together with many plane curves $A$ may be handled by piecewise appeal to Theorem~\ref{thm:main}. We will mention a few examples that highlight the importance of these results.

\begin{enumerate}
	\item We have convergence whenever $A$ is an algebraic curve and $B$ is a line. This includes the classical problems of ellipses and p-spheres \cite{BLSSS}. We also have the following results about the rate of convergence.
	\begin{enumerate}
		\item When $A$ and $B$ meet with multiplicity greater than one, as when $B$ is the horizontal axis and $A$ is the graph of $t \mapsto t^2$ (where $t\in\R$), we have linear convergence to the feasible point. This is in contrast with the setting of $DR$, where convergence is usually observed to be to a fixed point rather than a feasible point. Local convergence results about DR with plane curves have typically excluded such cases.
		\item When $A$ and $B$ meet with multiplicity one, as when $B$ is the horizontal axis and $A$ is the graph of $t \mapsto t^2-1$ (where $t\in\R$), we have superlinear convergence. When, additionally, the curves are not perpendicular at their point of meeting, the convergence is quadratic.
	\end{enumerate}
	\item We have local superlinear convergence for the case where $A$ is the graph of $t \mapsto t/\sqrt{|t|}$ \big(where $t\in\R\backslash \{0\}$\big) and $B$ is the horizontal axis, a case where the Newton--Raphson method cycles. This particular example is shown in Figure~\ref{fig:sqrtX}. Such cases highlight the importance of the projection onto the curve for preventing instability.
	\item\label{ex:colinear} As an example where the generically proper characterization of $C_T$ is useful, let $B=\R \times \{0\}$ and $A= {\rm gra}f$ with 
	\begin{equation*}
		f:t\mapsto \begin{cases}
			-t &\text{if}\;\;t \leq 1;\\
			-1 &\text{otherwise}.
		\end{cases}
	\end{equation*}
	Then, for any $x_n$ such that $P_A\vec{x}_n \in \{(t,-1)\;|\;t\geq1 \}$ has the same first coordinate as $\vec{x}_n$, one has a colinear case, whereupon $\vec{x}_{n+1}=T_{A,B}\vec{x}_n = \vec{x}_n+(0,1)$. Consequently, for $n$ sufficiently large, $P_A\vec{x}_n \in \{(t,-t)\;|\; t \in \R\}$, whereupon $\vec{x}_{n+1}=\CRM(\vec{x}_n)=(0,0) \in A \cap B$.
\end{enumerate}

\begin{figure}
	\begin{center}
		\includegraphics[width=0.4\textwidth]{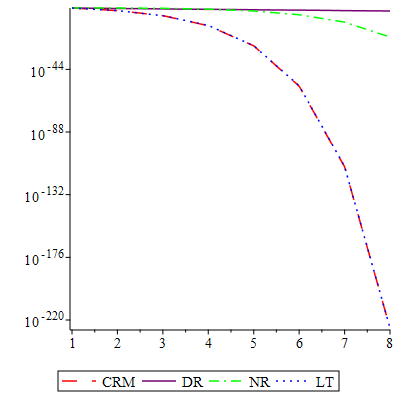}\quad \includegraphics[width=0.4\textwidth]{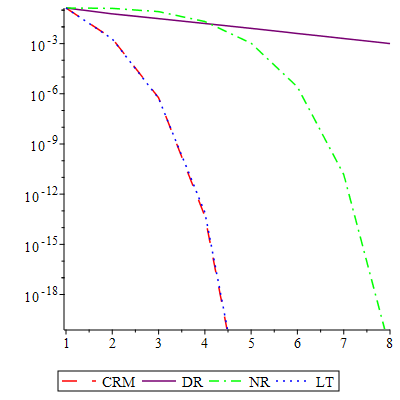}
	\end{center}
	\caption{Finding a point in the intersection of the sphere and line.}\label{fig:comparison}
\end{figure}

Succinctly put, for all cases when Douglas--Rachford exhibits local convergence to a \emph{feasible} point, CRM provides a better convergence rate. Finally, as a root finder, CRM has local convergence in all cases when Newton--Raphson does, exhibits quadratic convergence in all cases when Newton--Raphson does, and exhibits superlinear convergence in many cases when Newton--Raphson fails to converge at all. Figure~\ref{fig:comparison} shows the performance of the three methods, as measured by distance to solution, when $B$ is the horizontal axis and $A$ is the unit sphere centered at $(0,-1/2)$; the starting point used is $(0.9999,0)$. These results reveal a local robustness of CRM that may be prototypical of convergence for more general feasibility problems for more complicated problems of interest, such as phase retrieval. The method $L_T$ from \cite{lindstrom2020computable} is also shown; as its performance appears similar to CRM, it is worth remembering that the cost of computing an update of $L_T$ is double the cost of computing an update of CRM, because it requires the computation of 4 projections instead of 2.

\section{Rate guarantees and numerical discoveries in $\R^\eta$}\label{s:Rn}

While we were able to furnish Theorem~\ref{thm:subgradientdescent} in $\R^\eta$, the fast rate results of section~\ref{s:planecurves} are (initially) limited to $\R^2$. This restriction was somewhat unavoidable, because the framework we built was designed to be compatible with tying rate results explicitly to those of Newton--Raphson method. Interestingly, it is still possible to lift some of these results back into $\R^\eta$ for certain problems. Of course, given that superlinear and quadratic rate guarantees are \textit{very} strong, and also that they have been experimentally observed not to hold for many problems, one would expect that meaningful structural assumptions would be needed in order to guarantee a quadratic convergence result. 

Owing to complexity, a ``most general extension'' in $\R^\eta$ is not practicable. Indeed, even in $\R^2$, we reduced the statement of Theorem~\ref{thm:main} to just four cases (\ref{thm:concave}--\ref{thm:convex_zero}) for simplicity, but one could afford far greater generality by stating it in 16 cases (simply mix and match the properties of $|\theta|$ and $|\theta \circ (-\Id)|$ to obtain many more guarantees). 

What we will do in this section, then, is demonstrate one approach to adapting the theory from section~\ref{s:planecurves} for a selection of feasibility problems that have received significant attention in the literature: (1) spheres and hyperplanes, and (2) spheres and subspaces. Accordingly, we build one particular extension of our theory in $\R^\eta$ (Theorem~\ref{thm:spheres}) that is designed with the specific purpose of working for the examples we care about, and use this specific extension to provide rates for these problems. 

In the process, we stumble upon a much more significant discovery. Our numerical experiment for Example~\ref{ex:10sphereline} reveals an extreme sensitivity to numerical error that may cost CRM the superior convergence rates that are guaranteed by the theory. We correct for the numerical error, and recover the superior rate guaranteed by the theory. As the sensitivity is especially relevant to problem architectures that involve subspaces, all future works that use numerical experiment to study CRM with the product space formulation of feasibility problems (e.g. \cite{arefidamghani2021circumcentered,behling2019convex,dizon2020centering,DHL2020}) should take note of this sensitivity and account for it appropriately.

The particular extension theorem we  introduce will be easier to understand if we first present one of the examples that motivates it. Spheres and subspaces are of interest because they are prototypical of phase retrieval. They were studied for the Douglas--Rachford method in \cite{BS,AB,Benoist}, and the Lyapunov function discovered in \cite{Benoist} has been catalytic in other nonconvex investigations \cite{DT,giladi2019lyapunov}. Interestingly, global convergence for CRM for spheres and hyperplanes is implicitly shown in \cite[see~Remark~1]{behling2019convex}. We can furnish the first quadratic rate guarantees for these problems. For each of our computed examples, we will include the method $L_T$ from \cite{lindstrom2020computable}; for comparison, it is useful to remember that computing a step of $L_T$ has double the cost (4 projections instead of 2) of computing a step of CRM.

\begin{figure}[ht]
	\begin{center}
		\includegraphics[width=0.4\textwidth]{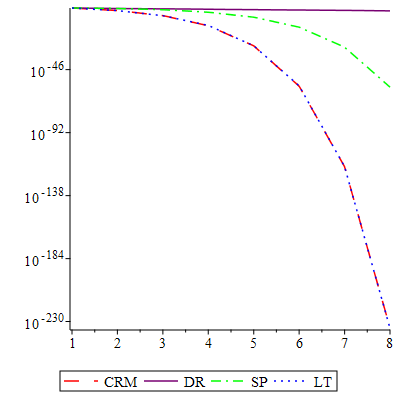}\quad \includegraphics[width=0.4\textwidth]{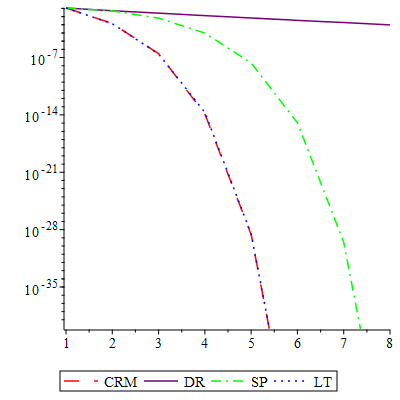}
	\end{center}
	\caption{Finding a point in the intersection of the 10-sphere and hyperplane from Example~\ref{ex:10spherehyperplane}.}\label{fig:10spherehyperplane}
\end{figure}

\begin{example}[Sphere and hyperplane]\label{ex:10spherehyperplane}
By symmetry, the (consistent) feasibility problem for a unit sphere $A$ and hyperplane not containing zero, $B$, in $\R^{\eta+1}$ may always be rotated so that $B:=\R^{\eta}\times \{b\}$ where $b \in \left]0,1\right]$. For $\vec{x}_0$ started in the axis of symmetry $\{0\}^{\eta} \times \R$, the Douglas--Rachford method is known to fail to converge, producing colinear reflections for all $n$ (see \cite{BDL18}).

For $\vec{x}_0$ started in $B$ and also outside of this chaotic set, it suffices by symmetry to consider $\vec{x}_0 =(0,\dots,0,t,b)$ for some $t \in \left]0,\infty \right[$. Suppose $\vec{x}_n=(0,\dots,0,t,b)$ for some $t \in \left]0,\infty \right[$. Because the projection of $\vec{x}_n$ onto $A$ is given by $\vec{x}_n/\|\vec{x}_n\|$, its reflection across the sphere is 
\begin{align*}
R_A\vec{x}_n &=2\vec{x}_n/\|\vec{x}_n\|-\vec{x}_n = (0,\dots,0,(2/\|\vec{x}_n\|-1)t,(2/\|\vec{x}_n\|-1)b )\\
\text{and so}\quad R_BR_A\vec{x}_n &=(0,\dots,0,(2/\|\vec{x}_n\|-1)t,2b-(2/\|\vec{x}_n\|-1)b)\\
\text{and so}\quad \CRM(\vec{x}_n) &\in \{0\}^{\eta-1} \times \R \times \{b\}\; =:D  \quad \text{(by Lemma~\ref{lem:alwaysindiagonal}\ref{p41d})}.
\end{align*}
Consequently, for all $n$, $\vec{x}_n \in \{0\}^{\eta-1} \times \R \times \{b\}$, while $P_A \vec{x}_n \in A \cap (\{0\}^{\eta-1}\times \R \times \R)$. In other words, the computation $\vec{x}_n$, including all of the associated projections and reflections, takes place entirely within a subspace of dimension $2$, and this subspace is the same for all $n$. The projections $P_A\vec{x}_n$ are projections onto the graph $\{(x,\Theta(x)) \in \R^{\eta+1} \;|\; x \in \R^\eta \}$ of the function $\Theta: \R^\eta \rightarrow \R: x \mapsto \sqrt{1-\|x\|_{\R^\eta}^2}$. More specifically, though, they are projections onto the slice of this graph that is given by $\{(x,\Theta(x)) \in \R^\eta \times \R \;|\; x \in \{0\}\times \dots \times \{0\}\times \R \times \{b\} \}$. This slice is the embedded lower-dimensional graph $\{\vec{v}+\theta(\vec{v})\cdot (0,\dots,0,1) \in \R^{\eta+1} \;|\; \vec{v} \in D \}$ of the function $\theta:D \rightarrow \R:\vec{v} \rightarrow \sqrt{1-|\vec{v}_\eta|^2}-b$. Moreover for all $n$, the point $\vec{s}: = (0,\dots,0, \sqrt{1-b^2},b)$ belongs to $A \cap B \cap {\rm aff}\{\vec{x}_n,R_A\vec{x}_n,R_BR_A\vec{x}_n \}$. With these conditions, Theorem~\ref{thm:spheres}, which we will introduce momentarily, guarantees that $\vec{x}_n$ converges to $\vec{s}$ with a quadratic rate (where the set $\{\vec{s}_n\}_n$ in the theorem is exactly equal to the singleton $\{\vec{s}\}$).

\textbf{Computed example}: In Figure~\ref{fig:10spherehyperplane} we show performance for a 10-sphere centered at $(0,\dots,0,-1/2)$. In this example, we include the method of subgradient projections for comparison, where we treat the first 9 variables as the domain space and the function for which we apply subgradient projections is $f:\vec{x}\mapsto \sqrt{1-(\vec{x}_1^2 + \dots + \vec{x}_9^2)}-1/2$. Of course, this algorithm is only defined on the $\mathbb{B}\setminus\{\vec{0}\}$ where $\mathbb{B}$ is the unit ball in $\R^9$. Moreover, for any dampened version of subgradient projections, if we start near enough to $\vec{0}$, our first update will be outside of this domain. Consequently, in order for subgradient projections to converge for this problem, we must start very close to the solution. For our experiments, we actually used CRM to obtain a starting point very near to the feasible set, and started all algorithms thereat. 
\end{example}

We now introduce our particular extension theorem, whose many conditions can all be verified for the sphere and subspace feasibility problems in Examples~\ref{ex:10spherehyperplane} and \ref{ex:10sphereline}. As we have already explained, many other generalizations in higher dimensions are possible, but Example~\ref{ex:10spherehyperplane} may be read as a template for motivating the specific conditions we chose to work with in this particular extension. 

\begin{theorem}\label{thm:spheres}
Let $A$ be a subset, and $B$ a proper subspace, of $\R^\eta$. Let $\vec{x}_0 \in B$ and $\vec{x}_n:=\CRM^n(\vec{x}_0)$. Suppose further that the following hold:
    \begin{enumerate}[label=(\roman*)]
        \item For all $n$, $\vec{x}_n \in B$;
        
        \item $\vec{x}_n \rightarrow \vec{y} \in B \cap A$ with $\xi>0$ satisfying $\vec{x}_n \in \mathbb{B}_{\xi}(\vec{y})$ for all $n$;
        
        \item There exists $\vec{u} \in B^\perp \cap \mathbb{B}_{1}(0)$ such that, for all $n$, $P_A(\vec{x}_n) = P_{{\rm gra} \Theta}(\vec{x}_n)$ where $\Theta: B \rightarrow \R$ and ${\rm gra}\Theta := \{\vec{x}+\Theta(\vec{x})\cdot \vec{u} \;|\; x \in B \}$; 
        
        \item For all $n$, the affine sets 
        $$
        \{r(\vec{x}_{n}-\vec{x}_{n+1}) \;|\;r \in \R \} \cap A \cap L \cap \mathbb{B}_{\xi}(\vec{x})
        $$
       are singletons, which we name $\vec{s}_{n+1}$ respectively;
       
        \item For the functions 
        $$
        \theta_{n}: r \mapsto \Theta \left(r\frac{(\vec{x}_{n}-\vec{s}_{n})}{\|\vec{x}_{n}-\vec{s}_{n}\|} +\vec{s}_{n}\right ),
        $$
        (which, as we have defined them, possess a root at $0$), it holds that $|\theta_n|$ and $|\theta_n \circ \Id|$ satisfy the basic conditions \ref{bc1}--\ref{bc5}. \item The functions $f_n=|\theta_n|$ satisfy \ref{case:convex} with corresponding $\epsilon_{f_n}>0$ and $\delta_{\rm RIGHT}^{f_n}>0$ (where $\delta_{\rm RIGHT}^{f_n}$ is as in Lemma~\ref{lem:NR}), as well as $\min\{\epsilon_{f_n},\delta_{\rm RIGHT}^{f_n} \}\geq \|\vec{x}_{n}-\vec{s}_{n}\|>0$ for all $n\geq N$ for some $N>0$.
    \end{enumerate}
    Then for all $n>N$, it holds that $\vec{x}_n$ satisfies
    \begin{equation*}
		\frac{\|\vec{x}_{n+1}-\vec{s}_n\|}{\|\vec{x}_n-\vec{s}_n\|^2} \leq \frac{M_{\rm RIGHT}}{2}.
	\end{equation*}
\end{theorem}
\begin{proof}
Suppose convergence is not finite. Let $n>N$. By a suitable translation, we can and do without loss of generality let $\vec{s}_n=\vec{0}$, which simplifies the notation and assures that $\|\vec{x}_n\| \leq \min\{\epsilon_{f_n},\delta_{\rm RIGHT}^{f_n} \}$. Define
$$
{\rm gra}\theta_n:=\{\vec{x}+\Theta(\vec{x}) \cdot \vec{u}\;|\; \vec{x} \in {\rm aff}\{\vec{x}_n,\vec{x}_{n+1} \}.
$$
Suppose for a contradiction that $P_A(\vec{x}_n) \notin {\rm gra}\theta_n$. Then, combining with the fact that $P_A(\vec{x}_n) \in {\rm gra}\Theta$, there exists $\vec{y} \in B \setminus {\rm aff}\{\vec{x}_n,\vec{x}_{n+1} \}$ with $P_A(\vec{x}_n)=\vec{y}+\Theta(\vec{y})\cdot \vec{u}$. Consequently, 
\begin{align}
    R_A(\vec{x}_n) &= 2\vec{y}+2\Theta(\vec{y})\cdot \vec{u} - \vec{x}\nonumber \\
    \text{and so}\quad R_BR_A (\vec{x}_n) &=2\vec{y}-2\Theta(\vec{y})\cdot \vec{u} - \vec{x}\nonumber \\
    \text{whereupon}\quad {\rm aff}\{\vec{x}_n,R_A\vec{x}_n,R_BR_A\vec{x}_n\} \cap B &= {\rm aff}\{2\vec{y}-\vec{x}_n,\vec{x}_n \}\nonumber \\
    &={\rm aff}\{\vec{y},\vec{x}_n \} \not\owns \vec{x}_{n+1},\label{eqn:contra}
\end{align}
where the final $\not\owns$ must hold because we assumed $\vec{y} \in B \setminus {\rm aff}\{\vec{x}_n,\vec{x}_{n+1} \}$. However, in view of the definition of CRM, it must hold that $\vec{x}_{n+1} \in {\rm aff}\{\vec{x}_n,R_A\vec{x}_n,R_BR_A\vec{x}_n\}$, while in view of Lemma~\ref{lem:alwaysindiagonal}\ref{p41d}, it must also hold that $\vec{x}_{n+1} \in B$. Consequently, 
$$
\vec{x}_{n+1} \in {\rm aff}\{\vec{x}_n,R_A\vec{x}_n,R_BR_A\vec{x}_n\} \cap B,
$$
and this contradicts \eqref{eqn:contra}. Thus we have shown that $P_A(\vec{x}_n) \in {\rm gra}\theta_n$. This fact combines with the fact that $P_A(\vec{x}_n)=P_{{\rm gra}\Theta}(\vec{x}_n)$ and the fact that ${\rm gra}\theta_n \subset {\rm gra} \Theta$, to guarantee that
$$
P_A(\vec{x}_n)=P_{{\rm gra}\theta_n}(\vec{x}_n).
$$
Therefore the entire CRM construction at step $n+1$ lives entirely in the 2-dimensional subspace ${\rm aff}\{\vec{x}_n,\vec{x}_{n+1} \} \oplus {\rm span}\{\vec{u}\}$. As we have $\vec{0} = \vec{s}_n \in {\rm aff}\{\vec{x}_n,\vec{x}_{n+1} \}$, we may use the change of coordinates: 
\begin{align*}
    B \oplus {\rm span}\{\vec{u}\}&\rightarrow \R \times \R: (\vec{x},\vec{x}) \mapsto \left(\langle \vec{x},\vec{x}_n \rangle / \|\vec{x}_n\|,\langle \vec{x},\vec{u} \rangle \right),
\end{align*}
whereupon $\vec{x}_n = (x_n,0)$ with $x_n = \|\vec{x}_n\| \leq \min\{\epsilon_{f_n},\delta_{\rm RIGHT}^{f_n} \} \leq \epsilon_{f_n}$. Thus we have all the necessary conditions to apply Lemma~\ref{lem:Newton--Raphson}\ref{eqn:NR}, whereupon $\|\vec{x}_{n+1}\| \leq g_{\theta_n}(\|\vec{x}_{n}\|)$. Combining with the fact that $\|\vec{x}_n\| \leq \min\{\epsilon_{f_n},\delta_{\rm RIGHT}^{f_n} \}$, Lemma~\ref{lem:NR} yields
	\begin{equation*}
		\frac{\|\vec{x}_{n+1}\|}{\|\vec{x}_n\|^2}  
		\leq \frac{g_{\theta_n}(\|\vec{x}_{n}\|)}{\|\vec{x}_n\|^2} < \frac{M_{\rm RIGHT}}{2}.
	\end{equation*}
\end{proof}

\begin{figure}[p]
	\hfill
	\subfloat[Starting with $x_0 \notin B$]{\includegraphics[width=0.49\linewidth]{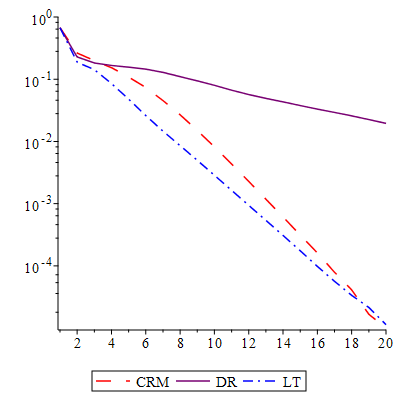}\label{fig:10spherelinea}}
	\hfill
	\subfloat[Starting with $x_0 \in B$]{\includegraphics[width=0.49\linewidth]{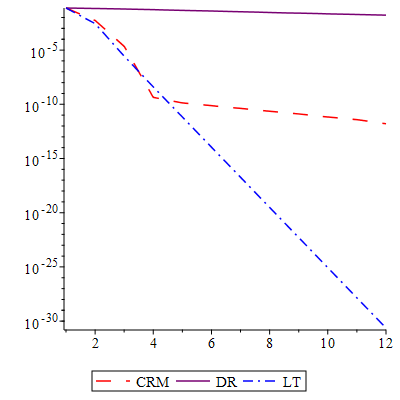}\label{fig:10spherelineb}}\hfill \\
	\begin{center}
	\subfloat[Numerically enforcing $x_n \in B$ for CRM and $L_T$]{\includegraphics[width=0.49\linewidth]{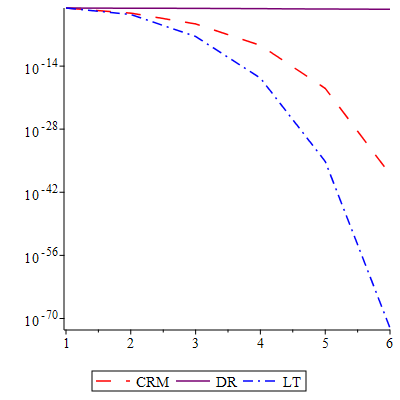}\label{fig:10spherelinec}}\end{center}
	\caption{Computational results from Example~\ref{ex:10sphereline}.}
\end{figure}

The more general subspace case in Example~\ref{ex:10sphereline} will be a straightforward adaptation of the hyperplane case in Example~\ref{ex:10spherehyperplane}. However, the phenomenon we observed in our numerical experiments is quite distinct from the hyperplane case. It demonstrates an extreme sensitivity of convergence rate for methods like CRM to small compounding numerical error that may cause sequences to depart from subspaces in which they are \textit{theoretically} guaranteed to remain. This small, numerically introduced, departure from the subspace results in a completely different convergence rate.

\begin{example}[A sphere and an affine subspace]\label{ex:10sphereline}

In $\R^{\eta+m}$, let $A$ be the unit sphere and $B$ be a subspace not containing zero. More specifically, let $B$ have dimension $\eta$. Then by symmetry, the (consistent) feasibility problem may always be rotated so that $B:=\R^\eta \times \{\vec{b}\}$ where $\vec{b} \in \left]0,1\right]^m$ and $\|\vec{b}\|_{\R^m} \leq 1$. We will consider the subset of cases where $\|\vec{b}\|_{\R^m} < 1$. For $\vec{x}_0$ started in the axis of symmetry $\{0\}^{\eta} \times \R^m$, the Douglas--Rachford method is known to fail to converge, producing colinear reflections for all $n$ (see \cite{BDL18}).

For $\vec{x}_0$ started in $B$ and also outside of this chaotic set, it suffices by symmetry to consider $\vec{x}_0 =(0,\dots,0,t,b_1,\dots,b_m)=:(0,\dots,0,t,\vec{b})$ for some $t \in \left]0,\infty \right[$. Suppose $\vec{x}_n=(0,\dots,0,t,\vec{b})$ for some $t \in \left]0,\infty \right[$. Because the projection of $\vec{x}_n$ onto $A$ is given by $\vec{x}_n/\|\vec{x}_n\|$, its reflection across the sphere is 
\begin{align*}
R_A\vec{x}_n &=2\vec{x}_n/\|\vec{x}_n\|-\vec{x}_n = (0,\dots,0,(2/\|\vec{x}_n\|-1)t,(2/\|\vec{x}_n\|-1)\vec{b})\\
\text{and so}\quad R_BR_A\vec{x}_n &=(0,\dots,0,(2/\|\vec{x}_n\|-1)t,2b-(2/\|\vec{x}_n\|-1)\vec{b})\\
\text{and so}\quad \CRM(\vec{x}_n) &\in \{0\}^{\eta-1}\times \R \times \{\vec{b}\}\; =:D . \quad \text{(by Lemma~\ref{lem:alwaysindiagonal}\ref{p41d})}
\end{align*}
Consequently, for all $n$, $\vec{x}_n \in \{0\}^{\eta-1}\times \R \times \{\vec{b}\}$, while $P_A \vec{x}_n \in A \cap (\{0\}^{\eta-1} \times \R \times {\rm span}\{\vec{b} \} )$. In other words, the computation $\vec{x}_n$, including all of the associated projections and reflections, takes place entirely within a subspace of dimension $2$, and this subspace is the same for all $n$. The projections $P_A\vec{x}_n$ are projections onto the graph $\{(x,\Theta(x)) \in \R^{\eta+m} \;|\; x \in \R^{\eta +m -1} \}$ of the function $\Theta: \R^{\eta+m-1} \rightarrow \R: x \mapsto \sqrt{1-\|x\|}$. More specifically, though, they are projections onto the slice of this graph that is given by $\{(x,\Theta(x)) \in \R^{\eta+m-1} \times \R \;|\; x \in \{0\}^{\eta-1}\times \R \times \{\vec{b}\} \}$. This slice is the embedded lower-dimensional graph $\{\vec{v}+\theta(\vec{v})\cdot (0,\dots,0,1) \in \R^{\eta+m} \;|\; \vec{v} \in D \}$ of the function $\theta:D \rightarrow \R:\vec{v} \rightarrow \sqrt{1-|\vec{v}_{\eta}|^2}-\|\vec{b}\|_{\R^m}$. Moreover $\vec{s}_n = (0,\dots,0, \sqrt{1-\|b\|_{\R^m}^2},\vec{b}) =:\vec{s}$ for all $n$. Applying Theorem~\ref{thm:spheres}, we have that $\vec{x}_n$ converges to $\vec{s}$ with a quadratic rate.

\textbf{Computed example}: For our subspace $B$, we used a line. As there is no longer a single obvious way to define subgradient projections for this problem, we omit it. In our first numerical experiment, we begin with a $x_0 \notin B$, and we observe apparently only linear convergence in Figure~\ref{fig:10spherelinea}. This is interesting, because it suggests that the condition $x_0 \in B$ cannot be relaxed without losing the quadratic rate. In our second experiment in Figure~\ref{fig:10spherelineb}, we begin with $x_0 \in B$ and see early apparent quadratic convergence degrade into only linear convergence. However, our theoretical results guarantee a quadratic rate. Investigating further, we discover that compounding numerical error has ultimately caused $\vec{x}_n$ to be outside of $B$, and this small difference has caused the loss of the superior convergence rate. For our final experiment in Figure~\ref{fig:10spherelinec}, we correct this compounding numerical error with an extra projection onto $B$; in other words, we are iteratively applying $\CRM \circ P_B$. Having corrected for the numerical error in this way, we see the quadratic convergence that the theory guarantees. This is quite interesting, because it illustrates the sensitivity of the convergence rate to very small numerical error. Purely for the sake of curiosity, in Figure~\ref{fig:10spherelinec} we also replace $L_T$ with $L_T \circ P_B$; in this case, the modification does not simply correct for numerical error, but actually defines a new algorithm. The apparent quadratic convergence of this algorithm is interesting and may merit future investigation.
\end{example}

\begin{figure}[ht]
	\hfill
	\subfloat[Partial solutions corresponding to 2,000th iterate]{\includegraphics[width=0.499\linewidth]{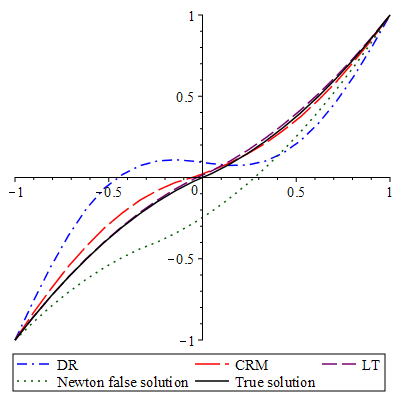}\label{fig:ODEsolutions}}
	\hfill
	\subfloat[Distance from discretized solution at $n$th iterate]{\includegraphics[width=0.499\linewidth]{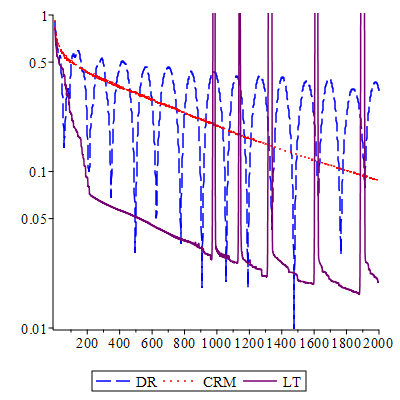}\label{fig:ODEperformance}}
	\hfill
	\caption{Computational results from Example~\ref{ex:ODE}.}
\end{figure}

For our next and final example, we choose an application that has \textit{some} similarities to the context we have worked in (one set is a subspace and the other an implicitly defined surface), but for which the rate is unknown. Naively, we are hoping to be pleasantly surprised by an improved rate, without necessarily expecting to be.

\begin{example}[A boundary value ordinary differential equation]\label{ex:ODE}

The task of numerically finding a discretized solution for a boundary value ordinary differential equation may be reformulated as a $2$-set feasibility problem, where one of the sets is an implicitly defined surface and the other is a subspace \cite{LLS,Lindstromthesis}. In \cite{LLS}, the  Douglas--Rachford method and method of alternating projections are applied to such problems. We test CRM and $L_T$ for the specific boundary value problem from \cite[Example~6.6]{LLS} (with a grid consisting of $20$ mesh points for $21$ segments; problem dimension $20^2$). The details for how to compute projections are described in \cite{LLS}. 

\textbf{Computed example}: We use the starting point that corresponds to a constant function that returns the value $0.5$ on the open interval between the boundary points (see \cite{LLS} for full details). A traditional solver---based on Newton's method---finds a false solution from this starting point, but each of the methods DR, CRM, and $L_T$ solved the problem in our experiment. Having learned the lesson about compounding numerical error from Example~\ref{ex:10sphereline}, we iterate $\CRM \circ P_B$, where $B$ is the \textit{agreement subspace} and the inclusion of $P_B$ serves to prevent compounding numerical error from pushing the sequence of updates out of $B$. We also replace $L_T$ with $L_T \circ P_B$.

Partial solutions obtained by DR, CRM, and $L_T$ (paused at 2,000 iterations), in addition to the true solution to the discretised problem and the false solution obtained by the traditional solver, are shown in Figure~\ref{fig:ODEsolutions}.  

In spite of the aforementioned similarities with our framework, we observe only linear convergence in Figure~\ref{fig:ODEperformance}. Here we record the norm distance between the shadow sequence of partial solutions for each method ($P_B \vec{x}_n$) and the true solution to the discretised problem. For a more detailed explanation of why error is reported this way, see \cite{LLS}. The progress of DR towards the solution is best measured by tracking the ``peaks'' of the visible ``tombstones'' in its performance profile; for more information about this, see \cite{LLS} or \cite{lindstrom2020computable}.

The fact that we observe only (apparently) linear convergence for CRM is unsurprising; it indicates that many conditions must all be met in order for CRM to have a superlinear convergence rate for a more general problem. When interpreting this graph, it is valuable to remember that one update of $L_T$ is twice as costly (requires 4 projections to be computed instead of 2) to compute as one update of $\CRM$. It is also worth noting that the case where $x_n, R_Ax_n, R_BR_Ax_n$ are distinct and colinear did not occur in our experiment.
\end{example}

\section{Conclusion}\label{s:conclusion}

The results in Sections~\ref{s:hypersurfaces} and \ref{s:planecurves} are natural analogs of those already in the literature for DR \cite{AB,BLSSS,BS,DT,LSS}. The framework, which relies on Theorem~\ref{thm:subgradientdescent}, is novel, and it illuminates a connection between CRM and subgradient projections in $\R^{\eta}$. In Section~\ref{s:Rn}, the roadmap for adaptation in $\R^\eta$, and the convergence rate guarantees for spheres and hyperplanes, while interesting in their own regards, are both overshadowed in importance by the numerical revelations of Example~\ref{ex:10sphereline}. Now that this numerical sensitivity has been documented, all future investigations of CRM on feasibility problems built with the traditional agreement subspace architecture \textit{must account for such numerical deviations}. If they do not do so, it is entirely possible that experiments will fail to reveal convergence rates that are theoretically achievable. The apparent quadratic convergence of the modified version of algorithm $L_T$ also clearly merits further investigation.

Even as recently as while this paper was in peer review, \cite{arefidamghani2021circumcentered} has used the epigraphical subgradient-descent characterization from Theorem~\ref{thm:subgradientdescent} to furnish rates for convex feasibility problems that are related to those here. For the product space formulation of a feasibility problem, one of the constraint sets is a subspace $B$. Under some local assumptions, (Lemma~\ref{lem:alwaysindiagonal}), subsequences $(\vec{x}_j)_j$ that have $\vec{x}_0 \in B$ satisfy $(\vec{x}_j)_j\subset B$, whereupon the 2-dimensional affine subspace ${\rm aff}(\vec{x}_j,R_A\vec{x}_j,R_BR_A\vec{x}_j) $ always contains a line $B \cap {\rm aff}(\vec{x}_j,R_A\vec{x}_j,R_BR_A\vec{x}_j)$. The rate guarantee from \cite{arefidamghani2021circumcentered}, when the angle between sets ``vanishes'' is related to Theorem~\ref{thm:main}\ref{thm:convex_zero}, while the superlinear guarantees when the angle does not vanish are related to Theorem~\ref{thm:main}\ref{thm:concave}--\ref{thm:convex}.

This is just one example of how these results are also of broader interest, because plane curve problems shed light on how such methods are thought to behave more generally. They motivated the introduction of 2 stage DR--CRM search algorithm in \cite{dizon2020centering,DHL2020}. Now that the groundwork has been laid for the prototypical settings, another natural next step of investigation is to conduct a detailed study on CRM for the phase retrieval problem specifically. Methods for accelerating convergence for Newton--Raphson are well known and are found in any numerical calculus textbook. The connections discussed in Section~\ref{s:convergence} indicate that another natural possible step is to attempt such methods with CRM. 

\subsubsection*{Data Availability}

The code used to generate the numerical results is available at \cite{DHLdata}.

\subsection*{Acknowledgements}
	SBL was supported in part by Hong Kong Research Grants Council PolyU153085/16p and by an AustMS Lift-Off Fellowship; his collaboration in this project was also made possible in part by funding from CARMA Priority Research Centre at University of Newcastle. JAH and NDD are supported by Australian Research Council Grant DP160101537.

%
%

\bibliographystyle{plain}      
\bibliography{refs}   

\end{document}